\documentclass{amsart}
\usepackage{graphicx}
\usepackage{amsmath, amssymb, amsthm}
\usepackage{tikz}
\usepackage{esint}
\usepackage{hyperref}
\usepackage{stmaryrd}
\usepackage{dirtytalk}
\usepackage{pgfplots}
\usepackage{float}
\usetikzlibrary{3d, perspective}
\pgfplotsset{compat=1.18} 
\allowdisplaybreaks
\hypersetup{
    colorlinks,
    citecolor=black,
    filecolor=black,
    linkcolor=black,
    urlcolor=black
}
\usetikzlibrary{arrows}
\newtheorem{theorem}{Theorem}

\newtheorem{lemma}{Lemma}

\newtheorem{corollary}{Corollary}

\newtheorem{example}{Example}
\DeclareMathOperator{\sgn}{sgn}

\newcommand{\NN}{\mathbb{N}}

\newcommand{\PP}{\mathbb{P}}

\newcommand{\RR}{\mathbb{R}}

\newcommand{\TT}{\mathbb{T}}

\newcommand{\ZZ}{\mathbb{Z}}
\newcommand{\A}{\mathcal{A}}

\newcommand{\C}{\mathcal{C}}
\newcommand{\D}{\mathcal{D}}

\newcommand{\F}{\mathcal{F}}

\newcommand{\I}{\mathcal{I}}

\renewcommand{\O}{\mathcal{O}}
\renewcommand{\P}{\mathcal{P}}

\newcommand{\Z}{\mathcal{Z}}

\title[PLE anti Hölder]{Uniform Positivity of Lyapunov Exponents for anti H\"older potentials}
\author[N. Chiem]{Nicholas Chiem}
\address{Department of Mathematics, University of California, Riverside, CA-92521, USA}
\email{nchie005@ucr.edu}

\begin{document}

\begin{abstract}
    We consider Schr\"odinger operators with dynamically defined potentials that are anti $\alpha$-H\"older and generated by hyperbolic dynamics. We prove an anti Lipschitz estimate on functions related to the orbit of $\vec{e}_1$ under an associated projectivized cocycle. After, we apply the estimate on subshifts of finite type and maps with one expanding direction on $\TT^d$ where $d\geq 1$. In particular, the application leads to uniformly positive Lyapunov exponents with sufficiently large coupling constants. 
\end{abstract}

\maketitle

\section{Introduction}

We are interested in studying the family of discrete Schr\"odinger operators with potentials generated by a compact metric space $(\Omega, T, \mu)$, where $T:\Omega\rightarrow \Omega$ is continuous and invertible and $\mu$ is a $T$-invariant measure. The one-dimensional Schr\"odinger operator is defined as $H_{\lambda,v,\omega}:\ell^2(\ZZ) \rightarrow \ell^2(\ZZ)$ where
\begin{equation}\label{d.sop}
    [H_{\lambda, v, \omega}\psi]_n = \psi_{n-1} + \psi_{n+1} + \lambda v(T^n\omega)\psi_n.
\end{equation}
In (\ref{d.sop}), $\omega\in \Omega$ is called the phase, $v: \Omega \rightarrow \RR$ is a bounded measurable function called the potential, and $\lambda\in \RR$ is called the coupling constant. If $T$ is not invertible we replace $\ZZ$ with $\NN$ and impose Dirichlet boundary condition: $\psi_{-1} = 0$.

Our interest in studying Schr\"odinger operators stems from a physical phenomenon discovered in 1958 by Phillip W. Anderson, Anderson localization. Localization takes on two forms: spectral and dynamical. We are particularly interested in spectral localization, which is when the spectrum of the operator is pure point with exponentially decaying eigenfunctions. One of the most well known cases is the Anderson model, which is when the potential is sampled by independent identically distributed (iid) random variables, each with a common distribution with support containing at least two distinct points. A special case is when the distribution has exactly two points: the associated operator family is called the Anderson-Bernoulli model. The latter setting is considered the most difficult and was initially proven by Carmona-Klein-Martinelli \cite{ckm} using multi-scale analysis. Later, it was also shown using one-dimensional tools by \cite{AvDaZh2, bdf+, gk, gz, jz, svw}. For more information and results on localization, see \cite{ak, DJ1, ds, gb, gek1, gek2, gek3, JK1, JK2, jz}.

To study localization, one explores the behavior of solutions from the eigenvalue equation $H_{\lambda, v, \omega} \psi = E\psi$ for energies $E\in \RR$. One direction to study the solutions is through Schr\"odinger cocycles $(T, A^{(E-\lambda v)})$, especially through the Lyapunov exponent. The Lyapunov exponent plays a central role in understanding localization; Kotani theory is a bridge between the two, as it relates the absolutely continuous spectrum to the essential closure of the set of energies where the Lyapunov exponent is zero. It has been observed that when paired with a large deviation estimate, one has a strong indication of localization, see \cite{jz, bg, BouSch, bdf+}.

We wish to point out that the method in the preceding paragraph studies the spectrum through a dynamical lens rather than a spectral lens. The latter has been explored by Damanik-Fillman \cite{DJ3, DJ4}. They show that for the doubling map and hyperbolic toral automorphisms with continuous sampling function, one has a connected essential spectrum. By using Kotani theory, one can return back to Lyapunov exponents and conclude positivity almost everywhere!

For our paper, we are interested in removing the entirety of the absolutely continuous spectrum by proving that the Lyapunov exponent is uniformly positive for all energies. The investigation of positive Lyapunov exponents has a rich history; the most famous settings established were the iid setting by Furstenberg \cite{fursten} and the quasi-periodic setting by Sorets-Spencer \cite{es} by extending methods from Herman \cite{herman}, and Bourgain \cite{Bou} for $d$-dimensional tori.  

Unfortunately, in alternative settings such as hyperbolic dynamics, the results are less gratifying. Independence plays a central role, thus by reducing to a strongly mixing system, one faces many complications. Our paper will focus on the hyperbolic setting, which contains two regimes: small and large coupling constants. For small coupling constants, Figotin-Pastur \cite{FigPas} developed methods that provided a necessary tool for Chulaevsky-Spencer \cite{ChuSpe} to show positivity away from the center and edges of the spectrum for deterministic potentials with small coupling constants. Building upon this, Bourgain-Schlag \cite{BouSch} established a large deviation theorem to achieve localization for almost every phase in the same spectral regime as Chulaevsky-Spencer \cite{ChuSpe}. For large couplings, Bjerklöv \cite{Bje} established positivity for circle endomorphisms and Zhang-Li \cite{ZhaLi} extends the result to expanding toral automorphisms. Moreover, J. Bourgain and E. Bourgain \cite{Bou2} established positivity for any non-negative coupling, as long as one has sufficient hyperbolicity and $C^1$ non-constant potentials. Results in this regime have also been shown by Sadel-Baldes \cite{SS},  Damanik-Killip \cite{DK}, and Krüger \cite{K}. Recently, a more complete answer for hyperbolic base dynamics was given from Avila-Damanik-Zhang \cite{AvDaZh}. They were able to show that for $\alpha$-H\"older continuous functions, the set $\Z = \{E\in \RR: L(E) = 0\}$ is discrete and there exists a constant $\lambda_0(v)$ so that $\Z$ is finite for all $0 < \lambda < \lambda_0$. Additionally, there is an open dense set $\O^\alpha\subset C^\alpha(\TT^d,\RR)$ such that one has uniformly positive Lyapunov exponents for all potentials in $\O^\alpha$. The technique used by Avila-Damanik-Zhang is the invariance principle, which has also been used to study positivity and continuity of the Lyapunov exponent, see \cite{BonVia,vy}. For more general information on results of positive Lyapunov exponents we refer to \cite{WanZha}. 

For our paper, a method of interest was developed by Young \cite{young}. Young studied cocycles of the form 
\begin{equation}\label{d.ycoc}
    A_\epsilon(\cdot) = \begin{pmatrix}
    \lambda & 0\\
    0 & \lambda^{-1}
    \end{pmatrix}R_{\phi_\epsilon(\cdot)}
\end{equation}
where $R$ is a rotation matrix that depends on a function $\phi_\epsilon: \RR/\ZZ \rightarrow \RR/2\pi\ZZ$. Young observed that positivity could be concluded so long as $\phi_\epsilon$ satisfies three conditions: compactly supported, monotone on the support, and derivative is positive and bounded away from zero. Damanik \cite[Problem 6]{Dam} observed that it would be possible to show positive Lyapunov exponents for Schr\"odinger cocycles with the method Young used. At the time, it was unclear how to apply it directly since the method required the cocycle to be of the form (\ref{d.ycoc}). Zhang showed that Schr\"odinger cocycles with sufficiently large coupling constants could be studied by defining a new cocycle which has a similar form of (\ref{d.ycoc}). The difference from (\ref{d.ycoc}) to Zhang's setting is that the matrix entries $\lambda$ are no longer constants, but functions. However, this can be dangerous as the functions may cancel the expanding effects of the base dynamics. One of Zhang's key observations was the function attached to the rotation matrix only needed two of the properties from Young's list: monotonicity and derivative uniformly bounded away from zero. This led to Zhang's result of positivity for potentials generated by the doubling map \cite{zhenghe2}.

Zhang's method for proving positivity followed a five step process, which is summarized below.
\begin{enumerate}
    \item[(i)] polar decomposition,
    \item[(ii)] control the derivatives of inductively defined functions $\theta_n = \phi_n + \theta \circ T^n$ from below,
    \item[(iii)] count discontinuities and points that achieve a $\frac{\pi}{2}+\pi\ZZ$ rotation,
    \item[(iv)] control the measure of pre-images of $\delta$-balls around achieved $\frac{\pi}{2}+ \pi\ZZ$ points,
    \item[(v)] conclude positivity.
\end{enumerate}
Briefly, step (i) has the purpose to move the Schr\"odinger cocycle to the form on (\ref{d.ycoc}). For step (ii), the functions $\phi_n$ and $\theta \circ T^n$ tracks the angle of the orbit of $\vec{e}_1$ after the action of the hyperbolic matrix and rotation, respectively. For concrete definitions of $\phi_n$ and $\theta\circ T^n$ we refer to section 2. Next, one counts the $\frac{\pi}{2}+\pi \ZZ$ rotations, as this is where the expanding effects of the hyperbolic matrix can be canceled out. In (iv), one ensures that even though $\frac{\pi}{2}+\pi\ZZ$ points arise, the amount of points close by is minimal to ensure expansion. Finally, positivity can be concluded using an argument similar to Young's method.

This method was proven to be flexible, as the method was used to show positivity for any hyperbolic $\mathrm{SL}(2,\RR)$ matrix acting on $\TT^2$; the adjustment from the $\TT$ to $\TT^2$ was to replace the uniform lower bound of the derivative with a uniform lower bound on the directional derivative in the unstable direction. 

Our paper further develops this method by providing a checklist of conditions to attain a suitable counterpart to (ii). We briefly move away from our original setup to prove a more general result which may be of independent interest. Instead, we let $\Omega$ be an ordered set with a self map $T:\Omega \rightarrow \Omega$. There are two key ingredients: monotone potential and $T$ satisfying a hyperbolic injective condition. The first is clear, but we wish to emphasize that the potential does not need to be differentiable anywhere. The latter condition  is as follows: there exists a sequence of partitions $\{\P_n\}$ such that for every $n>m$, then $\P_n$ is a refinement of $\P_m$. This means for every $S\in \P_n$, there exists an $S'\in \P_m$ such that $S\subset S'$. Moreover, for each $n\in \NN$, $T^n|_S$ injects into $\Omega$ for all $S\in \P_n$. When this condition is satisfied we say the map $T$ satisfies the \textit{hyperbolic injective condition}. We are now ready to state our key lemma.

\begin{lemma}\label{l.key}
    Let $\Omega$ be a space with order. Let $T:\Omega \rightarrow \Omega$ satisfy the hyperbolic injective condition. If $v: \Omega \rightarrow \RR$ is monotone and bounded, then there exists $\lambda_0 = \lambda_0(v)>0$ such that $\theta_n$ is monotone on every $S\in \P_n$ for all $n\in \NN$. Moreover, for any compact set $K\subset \RR$ there exists a constant $\eta = \eta(v,K) > 0$ such that for all $S\in \P_n$ and   $\omega,\omega'\in S$, one has the inequality
    \begin{equation}\label{e.l.monotone}
        |\theta_n(\omega',t) - \theta_n(\omega,t)|\geq \eta|v(T^n\omega') - v(T^n\omega)|,
    \end{equation}
    for all $t\in K$.
\end{lemma}

The proof of Lemma \ref{l.key} requires us to engage in precise estimates on the inductively defined functions. The overarching idea to show monotonicity comes from the geometry of the graph of $\theta_n$, particularly the interaction between the sum of $\phi_n$ and $\theta\circ T^n$. Geometrically, for each $S\in \P_n$, there is small subset $U\subset S$ where $\phi_n$ is expanding and monotone. On $S\setminus U$, $\phi_n$ is a contraction and is not necessarily monotone. On the other hand, $\theta\circ T^n|_S$ is monotone increasing on all of $S$. Thus, $\theta_n|_U$ is dominated by $\phi_n$ and, due to high expansion, controls the monotonicity. On $S\setminus U$, the high contraction of $\phi_n$ implies that $\phi_n$ contributes little to the sum, leading $\theta\circ T^n$ to dominate and monotonicity of $\theta_n$ on all of $S$.

Even though the geometry is clear, the analysis becomes much more complicated, particularly for the latter half of the lemma. Our functions need not be differentiable, but the estimates from the Mean Value Theorem hold because we obtain a uniform bound to replace the derivative of the function. The essential idea is to extract the uniform bound from $v$ having bounded image in $\RR$. 

Returning to our original setup, it is evident that Lemma \ref{l.key} provides a sufficient replacement to (ii), but it is not clear how monotonicity and the hyperbolic injective property imply positivity. The practical application of Lemma \ref{l.key} is when we impose further conditions on $v$, such as a anti Lipschitz condition. A function $v$ is anti Lipschitz if there exists a constant $L>0$ such that for any $\omega, \Tilde{\omega}\in \Omega$,
\[|v(\omega) -  v(\Tilde{\omega})|\geq Ld(\omega,\Tilde{\omega})\]
where $d$ is some metric on $\Omega$. In particular, Jitomirskaya and Kachkovskiy studied these potentials and were able to prove localization in the quasi-periodic setting, see \cite{JK1, JK2}.

This class of anti Lipschitz potentials does not contain all the possible potentials that can be utilized. Thus, we introduce the class of anti $\alpha$-H\"older potentials. That is, $v:\Omega \rightarrow \RR$ is anti $\alpha$-H\"older on $\Omega$, or globally anti $\alpha$-H\"older, if for some $\alpha\in (0,\infty)$ there exists an $H > 0$ such that for any $\omega,\Tilde{\omega}\in \Omega$,
\begin{equation}\label{d.lowerholder}
    |v(\omega)- v(\Tilde{\omega})|\geq Hd(\omega,\Tilde{\omega})^\alpha.
\end{equation} 
It is apparent that the new class of potentials contains all anti Lipschitz functions. We note that for some base dynamics $(\Omega, T)$, there may exist an $\alpha(\Omega)>0$ such that the set of anti-Hölder functions are empty for all $(0,\alpha]$. Thus, our result automatically holds when considering such $\alpha$. For examples of potentials in the class of anti $\alpha$-Holder, we refer to Application \hyperlink{Application2}{2}.

Provided a potential in the class of anti $\alpha$-H\"older and depending on the base dynamics, we are able to emulate Zhang's method to achieve positive Lyapunov exponents. To express the significance of Lemma \ref{l.key} we explore two applications: subshifts of finite type and $d$-dimensional tori with a map with at least one expanding direction.

For the following application sections, we leave out explicit definitions to introduce the main goals of each section. The definitions will be provided in the corresponding section where the proofs lie; for Application \hyperlink{Application1}{1} and \hyperlink{Application2}{2}, see Sections \hyperlink{prelim}{4.1} and \hyperlink{OEDT}{5}, respectively.

\subsection{Application 1: Subshift of finite type} \hypertarget{Application1}{}

For this section, let $\Omega$ represent a subshift of finite type and $T$ represent the left shift map. Let $\mu$ be a $T$-invariant probability measure with bounded distortion property and fully supported on $\Omega$. For each $\omega\in \Omega$, define a local unstable leaf $W^u_{loc}(\omega)$, which is the set of elements that agree with $\omega$ on components $n\leq 0$. For the subshifts in the half-line setting we write $\Omega^+, T_+, \mu^+$ with the same assumptions. For concrete definitions we refer to Section \hyperlink{prelim}{4.1}.

For our purposes we weaken the assumption of (\ref{d.lowerholder}) to uniformly anti $\alpha$-H\"older along local unstable sets. That is, for some $\alpha\in (0,\infty)$ there exists a uniform constant $H>0$ such that for any $\omega^-\in \Omega^-$ then for any $\omega, \Tilde{\omega}\in W^u_{loc}(\omega^-)$ 
\[|v(\omega)- v(\Tilde{\omega})|\geq Hd(\omega,\Tilde{\omega})^\alpha.\]
It is apparent that (\ref{d.lowerholder}) implies the above. The converse direction is false in general, as one may not be able to compare elements in different local unstable sets, see Example \ref{ex.3}. This leads us to our main theorem in this section.

\begin{theorem}\label{t.fullPLE}
    Consider the full-line subshift of finite type $(\Omega, T,\mu)$ as above. Let $v$ be uniformly anti $\alpha$-H\"older continuous, bounded, and monotone with respect to dictionary order all on $W^u_{loc}(\omega)$ for every $\omega\in \Omega$. Consider the family of Schr\"odinger operators as in (\ref{d.sop}). Then there exists a $C_0 = C_0(v)>0$ and $\kappa = \kappa(\mu)\geq 1$ such that for any $\lambda> 0$,
    \[L(E;\lambda) > \kappa^{-1} \log\lambda - C_0 \text{ for all } E\in \RR.\]
\end{theorem}

Before following the framework of (i) - (v), we must reduce the full-line to the half-line. A direct application of Rohklin's disintegration theorem leads to the half-line case. Hence it suffices to prove Theorem \ref{t.halfPLE} to prove Theorem \ref{t.fullPLE}. See Section \hyperlink{fullline}{4.3} for details.

\begin{theorem}\label{t.halfPLE}
    Consider the half-line subshift of finite type $(\Omega^+, T_+, \mu^+)$ as above. Let $v$ be anti $\alpha$-H\"older, continuous, bounded, and monotone with respect to dictionary order on $\Omega^+$. Consider the family of Schr\"odinger operators for the half-line. Then there exists a $C_1 = C_1(v)>0$ such that for any $\lambda> 0$,
    \[L(E;\lambda) > \log\lambda - C_1 \text{ for all } E\in \RR.\]
\end{theorem}

Observe that there is a $\kappa^{-1}$ constant difference between the theorems. The constant arises as one cannot directly apply Theorem \ref{t.halfPLE} on the disintegrated measure. Indeed, the reduction may lead to measures that lack $T_+$-invariance. By paying the price of a constant, one can pass to a $T_+$-invariant measure to apply Theorem \ref{t.halfPLE}.

For the half-line, steps (i), (iii), and (v) all follow by the arguments in \cite{zhenghe2}. Steps (ii) and (iv) contain essential difficulties. For (ii), there is no differentiable structure present for subshifts of finite type; the estimate in Lemma \ref{l.key} is utilized as the replacement of (ii). On the other hand, (iv) has two complications: relating $\delta$-balls around $\frac{\pi}{2} + \pi\ZZ$ points to measurable sets and estimating the measure of the sum of all such measurable sets. The first is primarily handled by the estimate from Lemma \ref{l.key}. The second is resolved through the bounded distortion property of $\mu$.

We conclude by remarking some well-known dynamical systems that are encapsulated in the result:
\begin{enumerate}
    \item Any subshift of finite type with a measure of maximal entropy.
    \item Bernoulli-shift on $n$-letters for $n\in \NN_{\geq 2}$ with Bernoulli measure.
    \item Topological Markov chains with Markov measure.
\end{enumerate}

\subsection{Application 2: One expanding direction on d-dimensional tori} \hypertarget{Application2}{}

We turn our attention to a more concrete class of systems. It is well known that hyperbolic dynamics can be encoded into shift space, but it is not true that the subshift with dictionary order respects the order of the original dynamics encoded. For example, Arnold's cat map and local unstable leaves, as defined in \cite{chiem}, have a significant ordering difference compared to the encoded local unstable sets. On the other hand, some systems do preserve order, most notably the doubling map on $\TT$. Even with order preserved, the set of anti $\alpha$-H\"older potentials on the encoded space can drastically differ from the set on the original space. Lemma \ref{l.key} enables us to prove positivity without considering the encoded space. In particular, we are able to prove positivity for maps with only one expanding direction! 

For any positive integer $d\geq 1$, define the set $G(d,\ZZ)$ to be the collection of integer valued $d\times d$ matrices with nonzero determinant and at least one real-eigenvalue $|\beta|\geq 2$. For notation, let $u$ represent the associated eigenvector to the eigenvalue $\beta$. Given a matrix $M\in G(d,\ZZ)$ we let $T_M$ represent the induced map on $\TT^d$ and when the context is clear we write $T$. We remark that Lebesgue measure is $T$-invariant, since $M$ has nonzero determinant. Similarly to the previous section we consider potentials that are uniformly anti $\alpha$-H\"older along local unstable leaves, but we only consider $\alpha\in [1,\infty)$. Briefly, local unstable leaves are line segments contained in the fundamental domain pointed in the direction $u$, see Section \hyperlink{OEDT}{5} for concrete details.

The main theorem of the section is the following.

\begin{theorem}\label{t.PLEhd}
    Consider any positive integer $d\geq 1$ and $M\in G(d,\ZZ)$. Consider the measurable system $(\TT^d, T_M, Leb)$ and a potential $v$ that is uniformly anti $\alpha$-H\"older, continuous, bounded, and monotone all along local unstable leaves. For the family of Schr\"odinger operators defined as in (\ref{d.sop}), there exists a constant $C_0 = C_0(v)>0$ such that for all $\lambda > 0$
    \[L(E; \lambda) > \log\lambda - C_0\]
    for all $E\in \RR$.
\end{theorem}

To prove Theorem \ref{t.PLEhd}, consider a positive integer $d\geq 2$. The first step is to reduce the problem to $1$-dimension, as steps (i) - (v) necessitate $1$-dimension. Analogous to Arnold's cat map, we will partition $\TT^d$ into union of line segments. The generalization is to find a set $Z$ of dimension $d-1$, contained in the fundamental domain $F$ of $\TT^d$, so that every $z\in Z$ corresponds to a unique line segment $W_z$. By Rohklin's disintegration theorem the problem is reduced to proving positivity on $W_z$. With this reduction, the argument on $W_z$ is the same that one would run for $d=1$. We then apply (i) with a minor difference by considering the convergence with respect to $C^0$-topology instead of $C^1$. After, we partition $W_z$ appropriately to apply Lemma \ref{l.key} to replace (ii). The remaining steps (iii) - (v) follow a similar argument as \cite{zhenghe2}. Consequently, Theorem \ref{t.PLEhd} extends results from \cite{zhenghe2, chiem}, which we formulate in two corollaries.

\begin{corollary}
    Consider the expanding map $(\TT, x\mapsto kx \mod 1, Leb)$ with integer $k\geq 2$. Let $v$ be a potential that is anti $\alpha$-H\"older, continuous, bounded, and monotone. For the family of Schr\"odinger operators defined as in (\ref{d.sop}), replacing $\ZZ$ with $\NN$, there exists a constant $C_0 = C_0(v) > 0$ such that for all $\lambda>0$
    \[L(E;\lambda) > \log\lambda - C_0\]
    for all $E\in \RR$.
\end{corollary}

\begin{corollary}
    Consider any hyperbolic $\mathrm{SL}(2,\ZZ)$ and consider $(\TT^2, T_M,  Leb)$. Let $v$ be a potential that is uniformly anti $\alpha$-H\"older, continuous, bounded, and monotone all along the local unstable leaves. For the family of Schr\"odinger operators (\ref{d.sop}), there exists a constant $C_0 = C_0(v)>0$ such that for all $\lambda > 0$
    \[L(E; \lambda) > \log\lambda - C_0\]
    for all $E\in \RR$.
\end{corollary}

The corollaries contain the family of $C^1$ potentials with derivative bounded away from zero, since such potentials are anti $1$-H\"older. Indeed, for any positive integer $d\geq 1$ consider any potential $v$ that is $C^1$ on the fundamental domain of $\TT^d$ and is monotone in direction $u$ with directional derivative bounded away from zero. That is, $|D_u v| > c$ for some constant $c$. In either case, $D_u v > c$ or $D_u v< -c$, from the Mean Value Theorem, we obtain
\[|v(\omega)- v(\Tilde{\omega})|\geq c ||\omega - \Tilde{\omega}||\]
where $||\cdot||$ represents $d$-dimensional Euclidean metric.

On the other hand, recall that the essential assumptions were monotonicity and derivative uniformly bounded away from zero. Our corollaries allows us to relax the latter condition to state derivative positive or negative. One benefit is that our potentials are allowed to admit critical points, so long as we have monotonicity and can verify the anti $\alpha$-H\"older condition. In practice, verifying the anti $\alpha$-H\"older condition can be complicated unless additional assumptions are imposed. This leads to a natural first set of examples which are a class of $C^r$ potentials. 

Before being able to show that a class of $C^r$ functions are anti $\alpha$-H\"older, we need a preliminary example. In particular, the example shows that piecewise anti H\"older functions, with each piece having possibly differing $\alpha$ constants, are globally anti $\alpha$-H\"older. Moreover, if $v$ is anti $\alpha$-H\"older, then $v$ is anti $\beta$-H\"older where $\beta\in [\alpha, \infty)$. For simplicity, all the examples provided assume monotone increasing, but this is not necessary as there is an analogous adjustment of assumptions for monotone decreasing.

\begin{example}[Piecewise anti H\"older functions are globally anti H\"older.]\label{ex.piece}
    Let $v: [0,1) \rightarrow \RR$ continuous, monotone increasing, and bounded. Let $a\in(0,1)$ and $\alpha_1, \alpha_2\in[1,\infty)$. Assume that $v|_{[0,a]}$ is anti $\alpha_1$-H\"older with constant $H_1>0$ and $v|_{[a,1)}$ is anti $\alpha_2$-H\"older with constant $H_2>0$. Then $v$ is anti $\alpha$-H\"older with  constant $H = 2^{1-\alpha}\cdot \min\{H_1, H_2\}$ where $\alpha = \max\{\alpha_1, \alpha_2\}$.

    Note that if $x,y\in [0,a)$ or $x,y\in [a,1)$ then there is nothing to show, since anti $\alpha_1$-H\"older and anti $\alpha_2$-H\"older imply $\alpha$-H\"older. It remains to consider $y\in [0,a)$ and $x\in [a,1)$. Observe,
    \begin{align*}
        v(x) - v(y) &= v(x) - v(a) + v(a) - v(y)\\
        &\geq H_1|x-a|^{\alpha_2} + H_2|a-y|^{\alpha_1}\\
        &\geq \min\{H_1, H_2\}\cdot \Big(|x-a|^{\alpha} + |a-y|^{\alpha}\Big)
    \end{align*}
    Notice that the function $x\mapsto x^\alpha$ is convex, so by convexity with $t= \frac{1}{2}$ and points $(x-a)$ and $(a-y),$, we find:
    \[2^{1-\alpha}|x-y|^\alpha\ \leq \Big(|x-a|^{\alpha} + |a-y|^{\alpha}\Big)\]
    It follows that
    \[|v(x) - v(y)|\geq H|x-y|^\alpha.\]
    Thus, $v$ is anti $\alpha$-H\"older.
\end{example}

Example \ref{ex.piece} can be generalized to hold for any finite number of pieces. In particular, if $v:[0,1) \rightarrow \RR$ satisfies the following: there exists $0 = x_0 < x_1 < \cdots < x_n = 1$ and $\alpha_0,\dots, \alpha_{n-1}\in [1,\infty)$ such that $v|_{[x_j,x_{j+1}]}$ is anti $\alpha_j$-H\"older for each $j = 0, \dots, n-1$. Then $v$ is $\alpha$-H\"older where $\alpha = \max_j \alpha_j$, which can be shown by using an inductive argument with the method as in Example \ref{ex.piece}. Notably, this provides a construction to create new anti $\alpha$-H\"older functions from a finite number of anti $\alpha$-H\"older functions. This idea is utilized to show that a class of $C^r$ potentials are anti $\alpha$-H\"older.

\begin{example}[$C^r$ potentials]\label{ex.Cr}
    Let $v\in C^{r}([0,1))$ for $r\geq 1$, monotone increasing, bounded, and exists some $k\in\NN$ so that $x_1,\dots, x_k$ has derivatives vanishing up to orders $n_1,\dots, n_k \in 2\NN$, respectively. Moreover, we assume for each $j = 1,\dots, k$, then $v^{(n_j + 1)}(x_j) > 0$. 
    
    We claim that for each $x_j$, there exists a $\delta_j>0$ such that $v|_{\overline{B(x_j,\delta_j)}}$ is anti $r$-H\"older. Letting $M_j = \max_{k} |v^{(k)}(x_j)|$, then it is apparent that 
    \[M_j\sum_{k=1}^{r-n_j-1} \Big|\frac{x^k}{k!}\Big| \rightarrow 0\]
    when $x \rightarrow 0$. Thus, we may pick $\delta_j>0$ so that
    \[\frac{v^{(n_j+1)}(x_j)}{(n_j+1)!}  - M_j\sum_{k=1}^{r-n_j-1} \Big|\frac{(y-x)^k}{k!}\Big| > \frac{c}{2(n_j+1)!},\]
    for any $y\in \overline{B(x_j,\delta_j)}$. By a direct application of Taylor's theorem and, if necessary, shrinking $\delta_j$. Then
    \begin{align*}
        |v(y) - v(x_j)| &= \Bigg|\sum_{k = 1}^r \frac{v^{(k)}(x_j)}{k!}(y-x_j)^k\Bigg| \\
        &= \Bigg|\frac{v^{(n_j+1)}(x_j)}{(n_j+1)!} (y-x_j)^{n_j+1} + \sum_{k = n_j+2}^r \frac{v^{(k)}(x_j)}{k!}(y-x_j)^k\Bigg|\\
        &= |(y-x_j)|^{n_j+1} \Bigg|\frac{v^{(n_j+1)}(x_j)}{(n_j+1)!} + \sum_{k = n_j+2}^r \frac{v^{(k)}(x_j)}{k!}(y-x_j)^k\Bigg|\\
        &\geq |y-x_j|^{n_j+1}\Bigg|\frac{v^{(n_j+1)}(x_j)}{(n_j+1)!}  - M_j\sum_{k=1}^{r-n_j-1} \Big|\frac{(y-x_j)^k}{k!}\Big|\Bigg|\\
        &\geq \frac{c}{2(n_j+1)!}|y-x_j|^{n_j+1} > 0,
    \end{align*}
    for any $y\in \overline{B(x_j,\delta_j)}$. For simplicity, let $x,y\in \overline{B(x_j,\delta_j)}$ and $y>x$. Using Taylor's theorem and convexity of $x\mapsto x^n$ for any $n\in\NN$, we obtain
    \begin{align*}
        v(y) - v(x) &= v(y) - v(x_j) + v(x_j) - v(x)\\
        &= \frac{v^{(n_j+1)}(x_j)}{(n_j+1)!}\Bigg((y-x_j)^r - (x - x_j)^r\Bigg) \\
        &+ \sum_{k = n_j+2}^r \frac{v^{(k)}(x_j)}{k!}\Bigg((y-x_j)^k - (x - x_j)^k\Bigg)\\
        &\geq 2^{-n_j}\frac{v^{(n_j+1)}(x_j)}{(n_j+1)!} (y-x)^{n_j+1} + \sum_{k = n_j+2}^r 2^{1-k}\frac{v^{(k)}(x_j)}{k!}(y-x)^k\\
        &= 2^{-n_j}(y-x)^{n_j+1}\Bigg(\frac{v^{(n_j+1)}(x_j)}{(n_j+1)!} + \sum_{k = n_j+2}^r 2^{1-k+n_j}\frac{v^{(k)}(x_j)}{k!}(y-x)^k \Bigg)\\
        &\geq 2^{-n_j}(y-x)^{n_j+1}\Bigg(\frac{v^{(n_j+1)}(x_j)}{(n_j+1)!}  - M_j\sum_{k=1}^{r-n_j-1} \Big|\frac{(y-x_j)^k}{k!}\Big|\Bigg)\\
        &\geq \frac{c}{2^{n_j}(n_j+1)!}(y-x)^{n_j + 1}
    \end{align*}
    The preceding calculation implies that $v|_{\overline{B(x_j,\delta_j)}}$ is anti $(n_j+1)$-Holder with constant $\frac{c}{2^{n_j}(n_j+1)!}$.
    
    On the other hand, without loss of generality, we may order
    \[x_1 < x_2 < \cdots < x_k.\]
    Define $I_0 = [0, x_1 - \delta]$, $I_i = [x_i+\delta, x_{i+1} - \delta]$ for $i = 1,\dots, k-1$, and $I_k = [x_k+\delta, 1]$. Monotone increasing implies $v'(x)\geq 0$ for all $x\in [0,1)$. As a consequence, if $x\in I_j$ for some $j = 1,\dots, k$, then $v'|_{I_j}(x) > c_j>0$ for some constant $c_j$. Similarly, this leads to $v|_{I_j}$ to be anti $1$-Holder with constant $c_j$. Thus, the potential is piecewise anti Holder and by Example \ref{ex.piece}, one can conclude that $v$ is anti $r$-H\"older.
\end{example}

Unfortunately, the orders of the derivatives that vanish must be even numbers, unless on the boundary. If not, then the function would locally behave like $x^2$ around the vanishing points, which destroys monotonicity. Indeed, assume that there existed a $v\in C^2([0,1))$ and a point $x_0$ not on the boundary such that $v^{(1)}(x_0) = 0$ and $v^{(2)}(x_0) > 0$. Consider $y_0 = x_0 -\delta$ for sufficiently small $\delta >0$, which exists since $x_0$ on not on the boundary. Then by the Fundamental Theorem of Calculus we have
\[- v^{(1)}(y_0) = \int_{y_0}^{x_0} v^{(2)}(x) dx > c \delta,\]
which directly implies that $v^{(1)}(y) < 0$. Thus, leading to a contradiction to the assumption that $v$ is monotone increasing. 

\begin{example}[Analytic potentials.]\label{ex.analytic}
    Let $v\in C^\omega(\RR/\ZZ, \RR/\ZZ)$ be strictly increasing, which can be lifted to a real analytic function and then restricted to $[0,1)$, so we denote this as $\Tilde{v}\in C^\omega([0,1),\RR)$. Moreover, $\Tilde{v}$ has finitely many critical points and each critical points has finite order. If each of the critical points vanish up to an even order, then with a direct application of the argument in Example \ref{ex.Cr} one can conclude $\Tilde{v}$ is anti $\alpha$-Hölder for some $\alpha\in [1,\infty)$.
\end{example}

The examples can be further generalized to higher dimension, as one can replace the derivative with a directional derivative in the unstable direction. To conclude the section, we discuss four more examples of potentials falling under Theorem \ref{t.PLEhd}. The first two are $C^1$ functions that can be defined in any dimension. The third is uniformly anti $\alpha$-H\"older along local unstable leaves, but not anti $\alpha$-H\"older. The final example shows that anti $\alpha$-H\"older is a larger class of potentials than anti Lipschitz and provides a simple example of potentials with critical points. 

\begin{example}\label{ex.exp}
     Consider any positive integer $d\geq 1$, $M\in G(d,\ZZ)$ with unstable unit eigenvector $u\in \RR^d$ and an $a\in \RR^d$ such that $\sum_{i=1}^d a_iv_i > 0$. Define
     \begin{align*}
         v: [0,1)^d &\rightarrow \RR\\
         (x_1,\dots, x_d) &\mapsto \exp\Big(\sum_{i=1}^d a_ix_i\Big).
     \end{align*}
     The directional derivative of $v$ in direction $u\in\RR^d$ is bounded below by $\sum_{i=1}^d a_iv_i$. The local unstable leaves would be the the image of the lines $f_z(t) = tv + z$ restricted to $[0,1)^d$ where $z\in [0,1)^d$, more compactly Image$(f_z)$. Consider any $x,y \in \RR^d$ such that $x,y\in$ Image$(f_z)$. By the Mean Value Theorem it follows that
     \[|v(x) - v(y)|\geq \sum_{i=1}^d a_iv_i ||x-y||.\]
     Hence, $v$ is uniformly anti $1$-H\"older along local unstable leaves.
\end{example}

\begin{example}
     Consider any positive integer $d\geq 1$, $M\in G(d,\ZZ)$ with unstable unit eigenvector $u\in \RR^d$ and an $a\in \RR^d$ such that $\sum_{i=1}^d a_iv_i > 0$. Define
     \begin{align*}
         v: [0,1)^d &\rightarrow \RR\\
         (x_1,\dots, x_n) &\mapsto \log(1+\sum_{i=1}^d a_ix_i)
     \end{align*}
     The directional derivative of $v$ in direction $u\in\RR^d$ is bounded below by \[\frac{\sum_{i=1}^d a_iv_i}{1+\sum_{i=1}^d a_i}.\] 
     Similarly to Example \ref{ex.exp}, the local unstable leaves would be the the graph of the lines $f_z(t) = tv + z$ restricted to $[0,1)^d$ where $z\in [0,1)^d$. Consider any $x,y \in \RR^d$ such that $x,y\in $Image$(f_z)$. By the Mean Value Theorem
     \[|v(x) - v(y)|\geq \frac{\sum_{i=1}^d a_iv_i}{1+\sum_{i=1}^d a_i}||x-y||.\]
     Hence, $v$ is uniformly anti $1$-H\"older along local unstable leaves.
\end{example}

\begin{example}[Uniformly anti $\alpha$-H\"older only on local unstable leaves]\label{ex.3}
    For simplicity, consider the matrix 
    \[M = \begin{pmatrix}
        2 & 0\\
        0 & 1
    \end{pmatrix}\]
    acting on $\TT^2$. The local unstable leaves are the graphs of the function $y = z$ for a constant $z\in [0,1)$. Define the potential
    \begin{align*}
        v:[0,1)^2 &\rightarrow \RR\\
        (x,y) &\mapsto (y-\frac{1}{2})^2 + x.
    \end{align*}
    Now consider any $(x, y), (\Tilde{x},y)\in [0,1)^2$ on the same local unstable leaf. It is immediately clear that
    \[|v(x,y) - v(\Tilde{x},y)| = |x - y| = d(x,y).\]
    On the other hand, consider $(\frac{1}{2}, \frac{1}{4})$ and $(\frac{1}{2}, \frac{3}{4})$, then it follows
    \[\Bigg|v\Big(\frac{1}{2}, \frac{1}{4}\Big) - v\Big(\frac{1}{2}, \frac{3}{4}\Big)\Bigg| = 0 < \Bigg|\Bigg|\Big(\frac{1}{2}, \frac{1}{4}\Big) - \Big(\frac{1}{2}, \frac{3}{4}\Big)\Bigg|\Bigg|.\]
    Implying that $v$ does not satisfy (\ref{d.lowerholder}).
\end{example}

\begin{example}[Anti $\alpha$-H\"older, but not anti Lipschitz] \label{ex.monomial}
    Consider any any positive integer $n\geq 1$ and define
    \begin{align*}
        v: [0,1) &\rightarrow \RR\\
        x &\mapsto x^n.
    \end{align*}
    Note that $v$ is continuous, bounded, and monotone increasing. To see that $v$ is not anti Lipschitz, assume for a contradiction that there exists a constant $L>0$ that satisfies the anti Lipschitz condition. The condition implies $L\leq \frac{1}{j^{n-1}}$ for all any positive integers $j\geq 1$, which is a contradiction. On the other hand, $v$ is anti $n$-H\"older, since for $a,b\in (0,1)$,
    \[a^n + b^n \leq (a+b)^n.\]
    Let $x,y\in (0,1)$ where $x>y$. Replacing $a = x-y$ and $b=y$, it follows that
    \[(x-y)^n \leq x^n - y^n.\]
    Hence, $v$ is anti $n$-H\"older.
\end{example}

The rest of the paper is organized as follows. Section \hyperlink{secprelim}{2} will introduce the preliminary definitions needed for all sections and a brief discussion of polar decomposition of Schr\"odinger cocycles. In Section \hyperlink{pfkey}{3}, we will prove Lemma \ref{l.key}. Finally, Sections \hypertarget{SFT}{4} and \hyperlink{OEDT}{5} will prove the results from Application \hyperlink{Application1}{1} and \hyperlink{Application2}{2}, respectively.

\subsection*{Acknowledgments} Thank you to my advisor Zhenghe Zhang for providing this exciting problem, as well as: pointing out areas of essential difficulties, extensive feedback to improve the structure of the paper, and the endless belief and support in me. Thank you to Adeline Chin for helping with the editing and readability of the paper. Thank you to Agnieszka Zelerowicz, Alex Tao, Hamid Naderiyan, Seth Berman, Ethan Nowaski, Kevin Costello, and Clark Butler for allowing me to discuss my research with all of you. Everyone has provided me support and/or guidance in some way and I am extremely grateful to everyone. 

\section{Preliminaries}\hypertarget{secprelim}{}

Recall we consider a compact metric space $(\Omega, T, \mu)$ with a continuous and invertible map $T$ and a measure $\mu$ that is $T$-ergodic. We are interested in studying Schr\"odinger operators acting on $\ell^2(\ZZ)$, as in (\ref{d.sop}). In the case that $T$ is not invertible, we replace $\ZZ$ with $\NN$ and impose Dirichlet boundary condition $\psi_{-1} = 0$.

We are interested in the behavior of $H_{\lambda,v, \omega} \psi = E \psi$ for $E\in \RR$. For each energy $E\in \RR$, define an associated Schr\"odinger cocycle $(T,A^{(E-\lambda v)})$. The cocycle map $A^{(E-\lambda v)}: \Omega \rightarrow \mathrm{SL}(2,\RR)$ is defined as
\begin{equation}\label{d.sc}
    A^{(E-\lambda v)}(\omega) = \begin{pmatrix}
    E - \lambda v(\omega) & -1\\
    1 &0
\end{pmatrix}.
\end{equation}
The iterates of the cocycle map are written as
\[A^{(E-\lambda v)}_n(\omega) = \begin{cases}
    A^{(E-\lambda v)}(T^{n-1}\omega) \cdots A^{(E-\lambda v)}(\omega) & n > 0\\
    I_2 & n = 0\\
    A^{(E-\lambda v)}(T^{n}\omega)^{-1} \cdots A^{(E-\lambda v)}(T^{-1}\omega)^{-1} & n< 0
\end{cases}\]
and have a direct relation to solutions of the eigenvalue equation, since $\psi$ is a solution, not necessarily in $\ell^2(\ZZ)$, if and only if 
\[\begin{pmatrix}
    \psi_{n}\\
    \psi_{n-1}
\end{pmatrix} = A^{(E-\lambda v)}_n \begin{pmatrix}
    \psi_0\\
    \psi_{-1}
\end{pmatrix} \text{ for all } n\in \ZZ.\]

The Lyapunov exponent is defined as
\begin{align}
    L(E; \lambda) &= \lim_{n\rightarrow \infty} \frac{1}{n}\int_{\TT^2} \log||A^{(E-\lambda v)}_n(\omega)|| d\mu(\omega)\\
    &= \inf_{n\geq 1} \frac{1}{n}\int_{\TT^2} \log||A^{(E-\lambda v)}_n(\omega)||  \nonumber \\
    &\geq 0 \nonumber
\end{align}
where the existence of the limit follows by subadditivity. Moreover, if $\mu$ is ergodic, then by Kingman's Subadditive Ergodic Theorem on each $E\in \RR$ to obtain
\[L(E;\lambda) = \lim_{n\rightarrow \infty}\frac{1}{n}\log||A^{(E-\lambda v)}_n(\omega)||\text{ for a.e } \omega\in \TT^2.\]
For more information on why the Lyapunov exponent is interesting, general theory, and specific classes on Schr\"odinger operators, we highly recommend, respectively \cite{wilkinson, DJ1, DJ2}. 

Throughout the remainder of the paper let $0<c<1<C$ represent universal constants that are small and large respectively. These constants only depend on the bounded, continuous, monotone, and anti $\alpha$-H\"older potential $v:\Omega \rightarrow \RR$. Without loss of generality, we will always assume that $0\leq v(\omega)\leq 1$ for all $\omega\in \Omega$ except for the proof of Lemma \ref{l.key}. 

Next, we briefly discuss polar decomposition, as it is an essential starting point, see \cite{zhenghe2} for a more detailed account. Let $t = \frac{E}{\lambda}$ and $t\in \I := [-1, 2]$. Define functions $r,g: \Omega\times \I \rightarrow \RR$ as
\begin{align*}
    r(\omega, t) &= t - v(\omega),\\
    g(\omega, t) &= r^2(\omega) + 1.
\end{align*}
Since $v$ is bounded we have
\[c < g(\omega,t) < C \text{ for all } \Omega\times \I.\]
We define a new matrix
\begin{equation}\label{e.coc}
    A(\omega, t, \lambda):=  \begin{pmatrix}
    \lambda \sqrt{g(T\omega, t)} & 0 \\
    0 & \frac{1}{\lambda} \sqrt{\frac{1}{g(T\omega, t)}}
\end{pmatrix}R_{\theta(\omega, t)}
\end{equation}
where we let $\Tilde{\Lambda}(T\omega)$ represent the hyperbolic matrix and
\[R_\gamma = \begin{pmatrix}
    \cos\gamma & -\sin\gamma\\
    \sin\gamma & \cos\gamma
\end{pmatrix}.\]

We define inductively defined functions from $\Omega\rightarrow \RR$ that follow the orbit of $\vec{e}_1\in \RR^2$ under the projective action of the the cocycle $A$. Define
\[\theta(\omega, t) = \cot^{-1}(t-v(\omega))\]
and let $\theta_0 = \theta(\omega,t)$. For every $n\in \NN$, we define
\[\phi_n(\omega, t) = \cot^{-1}(\lambda^2 g(T^n\omega)\cot\theta_{n-1}) \text{ and } \theta_n(\omega, t) = \phi_n(\omega, t) + \theta(T^n\omega).\]
Lastly, we denote the Lyapunov exponent for the cocycle $(T, A(\cdot, t, \lambda))$ as $\Tilde{L}(t;\lambda)$. 

The relationship between the cocycles $A$ and $A^{(E-\lambda v)}$ is not immediately clear. In the construction of $A$, we found that $A^{(E-\lambda v)}$ is a $\log$-integrable $\mathrm{SL}(2, \RR)$ conjugacy away from a matrix $\Lambda\cdot O$ where $\Lambda$ represents a diagonal matrix and $O$ is a rotation matrix both depending on $\lambda$. As $\lambda$ tends to infinity, the conjugated matrix $\Lambda\cdot O$ approaches $A\circ T^{-1}$ in the $C^0$-topology. By proving that $\Tilde{L}(t;\lambda)$ is positive, we may pass positivity and the lower estimates to the cocycle $(T,\Lambda\cdot O)$ by $T$-invariance and the stability of our argument under $C^0$ perturbations. Moreover, the Lyapunov exponent is preserved under $\log$-integrable conjugation, so positivity and lower estimates for $(T,\Lambda\cdot O)$ imply that the same holds for $(T,A^{(E-\lambda v)})$. In the non-invertible case, one instead shows $A^{(E-\lambda v)} \circ T$ is a conjugacy away from $(\Lambda \circ T) \cdot O$ and the rest follows an analogous argument. For additional details on polar decomposition we refer to \cite{zhenghe1} and \cite[Appendix A]{zhenghe2}.

\section{Proof of Key Lemma}\hypertarget{pfkey}{}

In this section, we prove Lemma \ref{l.key}, which will be an essential tool for us to show positivity. The argument is stable in $C^0$-perturbations, since the only dependence is on the shape of $\theta$ and $g$.

\begin{proof}[Proof of Lemma \ref{l.key}]
    Consider any compact set $K\subset \RR$, since $K$ must be closed and bounded there must exist some connected and compact interval containing $K$. Thus, without loss of generality, assume that $K = [-m_1, m_2]$ for some $0\leq m_1, m_2\in \RR$. Define 
    \begin{align*}
        M &= \max\{m_1, m_2\},\\
        G &= M + ||v||_\infty,\\
        \eta &= \frac{1}{2(1+G^2)},
    \end{align*}
    and
    \[\lambda_0 = \eta^{-1}(5||v||_\infty+2M).\]
    Fix any $\lambda\geq \lambda_0$. All estimates provided will be independent of $\lambda$ and $t\in K$. 
    
    Let $n=0$ and $\omega,\omega'\in \Omega$ such that $\omega<\omega'$. By monotonicity of $v$ and $\cot^{-1}$ being a decreasing function, we may conclude
    \[\theta_0(\omega) = \cot^{-1}(t-v(\omega)) < \cot^{-1}(t-v(\omega')) = \theta_0(\omega').\] 
    Thus, $\theta_0$ is monotone on $\Omega$.

    To show (\ref{e.l.monotone}), consider the function $\cot^{-1}: [-G, G] \rightarrow \RR$. By the Mean Value Theorem, for any $x, y\in [-G, G]$ with $x < y$, there is a $z\in (x,y)\subset [-G, G]$ such that
    \begin{equation}\label{e.mvtcot}
        \cot^{-1}(x) - \cot^{-1}(y) = \frac{1}{1+z^2}(y-x)\geq \eta(y-x).
    \end{equation}
    Remark that $t-v(\omega) \in [-G, G]$ for all $\omega\in \Omega$. Thus, we may replace $x,y$ with $t - v(\omega'),t-v(\omega)$, so
    \[\theta_0(\omega') - \theta_0(\omega) = \cot^{-1}(t-v(\omega')) - \cot^{-1}(t-v(\omega))\geq \eta(v(\omega')-v(\omega)).\]

    For induction, assume for some positive integer $n\in \NN$ that $\theta_{n}$ is monotone on any $S'\in \P_{n}$ and satisfies (\ref{e.l.monotone}). Consider any $S\in \P_{n+1}$ and let $\omega,\omega'\in S$ such that $\omega < \omega'$. By definition of $\P_{n+1}$ being a refinement of $\P_{n}$, there exists some $S'\in \P_{n}$ such that $S\subset S'$. By inductive hypothesis, $\theta_{n}$ must be monotone on $S'$ and hence $S$ too. Monotonicity of $v$, $T^{n+1}|_S$ injective, and $\cot^{-1}$ decreasing, imply that $\theta\circ T^{n+1}|_S$ is monotone. 
    
    Let $\alpha\in \NN$ such that $\theta_{n}(\omega')\in \Big(\alpha\pi, (\alpha+1)\pi \Big)$, then $\phi_{n+1}(\omega')\in \Big(\alpha\pi, (\alpha+1)\pi \Big)$. Similarly, if $\theta_{n}(\omega) \in \Big(k\pi, (k+1)\pi\Big)$ for some $k<\alpha$, then $\phi_{n+1}(\omega) < \alpha\pi$. The preceding line implies $\phi_{n+1}(\omega) < \phi_{n+1}(\omega')$ and hence $\theta_{n+1}(\omega) < \theta_{n+1}(\omega')$. The estimate follows by (\ref{e.mvtcot}).
    \begin{align*}
        \theta_{n}(\omega') - \theta_n(\omega)&\geq \theta_0(T^n\omega') - \theta_0(T^n\omega)\\
        &\geq \eta(v(T^n\omega') - v(T^n\omega))
    \end{align*}
    It remains to handle when
    \begin{equation}\label{e.samecomp}
        \theta_{n}(\omega), \theta_{n}(\omega')\in \Big(\alpha\pi, (\alpha+1)\pi \Big).
    \end{equation}
    If $\phi_{n+1}(\omega) < \phi_{n+1}(\omega')$, then the argument above completes this case. Thus, we assume
    \begin{equation}\label{e.xgreaty}
        \phi_{n+1}(\omega) > \phi_{n+1}(\omega').
    \end{equation}
    
    Observe that (\ref{e.samecomp}) implies $\cot\theta_{n+1}(\omega)$ and $\cot\theta_{n+1}(\omega')$ lie in the same $\pi \ZZ$ interval. The function $\cot^{-1}$ is Lipschitz, so
    \[0 < |\phi_{n+1}(\omega) - \phi_{n+1}(\omega')| \leq \lambda^2g(T^{n+1}\omega')\cot\theta_{n}(\omega') - \lambda^2g(T^{n+1}\omega)\cot\theta_{n}(\omega)\]
    By injectivity, we have $\theta(T^{n+1}S) \subset \theta(\Omega)$ and by (\ref{e.mvtcot}) we obtain
    \[\theta(T^{n+1}\omega') - \theta(T^{n+1}\omega)\geq 2\eta\Big(v(T^{n+1}\omega') - v(T^{n+1}\omega)\Big)\]
    for all $\omega,\omega'\in S$.

    On the other hand, assumption (\ref{e.xgreaty}) implies 
    \begin{equation}\label{e.gxlessgy}
        g(T^{n+1}\omega)\cot\theta_{n}(\omega) < g(T^{n+1}\omega')\cot\theta_{n}(\omega').
    \end{equation}
    Consider two cases:
    \begin{enumerate}
        \item $g(T^{n+1}\omega) < g(T^{n+1}\omega')$
        \item $g(T^{n+1}\omega) > g(T^{n+1}\omega')$
    \end{enumerate}
    These cases determine the sign of $\cot\theta_{n}(\omega)$ and $\cot\theta_{n}(\omega)$. In case 1,
    \begin{align*}
        0 &< g(T^{n+1}\omega')\cot(\theta_{n}(\omega')) - g(T^{n+1}\omega)\cot(\theta_{n}(\omega))\\
        &= \underbrace{g(T^{n+1}\omega')}_{\geq 1}\Big(\underbrace{\cot(\theta_{n}(\omega')) - \cot(\theta_{n}(\omega))}_{<0}\Big) + \cot(\theta_{n}(\omega))\Big(\underbrace{g(T^{n+1}\omega') - g(T^{n+1}\omega)}_{>0}\Big)
    \end{align*}
    which implies $\cot(\theta_{n}(\omega)) > 0$ and $\cot(\theta_n(\omega'))>0$. In case 2,
    \begin{align*}
        0 &> g(T^{n+1}\omega)\cot(\theta_{n}(\omega)) - g(\omega')\cot(\theta_{n}(\omega'))\\
        &= \underbrace{g(T^{n+1}\omega)}_{\geq 1}\Big(\underbrace{\cot(\theta_{n}(\omega)) - \cot(\theta_{n}(\omega'))}_{>0}\Big) + \cot(\theta_{n}(\omega')) \Big(\underbrace{g(T^{n+1}\omega) - g(T^{n+1}\omega')}_{>0}\Big)
    \end{align*}
    which implies that $\cot\theta_{n}(\omega') < 0$ and $\cot(\theta_n(\omega))>0$. 
    
    Consider the following subcases of the first case:
    \begin{enumerate}
        \item[(a)] $g(T^{n+1}\omega')\cot\theta_{n}(\omega') \geq \frac{1}{\lambda}$
        \item[(b)] $0 < g(T^{n+1}\omega')\cot\theta_{n}(\omega') < \frac{1}{\lambda}$
    \end{enumerate}
    \textit{Subcase a:} Define the function 
    \begin{align*}
        f(x):(0, \infty) &\rightarrow \RR\\
        x &\mapsto \cot^{-1}(x) - \frac{1}{x}
    \end{align*}
    and notice that $f$ is an increasing function. If $x<y$, one has 
    \begin{equation}\label{e.f}
        \cot^{-1}(x) - \cot^{-1}(y) < \frac{1}{x} - \frac{1}{y}.
    \end{equation}
    Replacing the $x,y$ in (\ref{e.f}) with $\lambda^2g(T^{n+1}\omega)\cot(\theta_{n}(\omega)), \lambda^2 g(T^{n+1}\omega')\cot(\theta_{n}(\omega'))$ and using the assumptions of our subcase we find
    \begin{align*}
        0 &< |\phi_{n+1}(\omega) - \phi_{n+1}(\omega')| \\
        &= |\cot^{-1}(\lambda^2g(T^{n+1}\omega)\cot(\theta_{n}(\omega))) -\cot^{-1}(\lambda^2g(T^{n+1}\omega')\cot(\theta_{n}(\omega'))|\\
        &< \Bigg|\frac{1}{\lambda^2g(T^{n+1}\omega)\cot(\theta_{n}(\omega))} - \frac{1}{\lambda^2g(T^{n+1}\omega')\cot(\theta_{n}(\omega'))}\Bigg|\\
        &= \Bigg|\frac{\lambda^2g(T^{n+1}\omega')\cot(\theta_{n}(\omega')) - \lambda^2g(T^{n+1}\omega)\cot(\theta_{n}(\omega))}{\lambda^4g(T^{n+1}\omega)g(T^{n+1}\omega')\cot\theta_{n}(\omega)\cot\theta_{n}(\omega')}\Bigg|\\
        &<\Bigg|\frac{\lambda^2g(T^{n+1}\omega')\cot(\theta_{n}(\omega)) - \lambda^2g(T^{n+1}\omega)\cot(\theta_{n}(\omega))}{\lambda^4 \cot\theta_{n}(\omega)g(T^{n+1}\omega')\cot\theta_{n}(\omega')}\Bigg|\\
        &= \Bigg|\frac{\lambda^2\cot(\theta_{n}(\omega))\Big(g(T^{n+1}\omega') - g(T^{n+1}\omega)\Big)}{\lambda^4 \cot\theta_{n}(\omega)g(T^{n+1}\omega')\cot\theta_{n}(\omega')}\Bigg|\\
        &\leq \Bigg|\frac{1}{\lambda} \Big(g(T^{n+1}\omega') - g(T^{n+1}\omega)\Big)\Bigg|\\
        &= \Bigg|\frac{1}{\lambda}\Big(v(T^{n+1}\omega') - v(T^{n+1}\omega)\Big) \Big(v(T^{n+1}\omega') + v(T^{n+1}\omega) - 2t\Big)\Bigg|\\
        &< \frac{5||v||_\infty + 2M}{\lambda}\Big(v(T^{n+1}\omega') - v(T^{n+1}\omega)\Big).
    \end{align*}
    By our choice of $\lambda_0$, it follows that
    \begin{align}
        \theta_{n+1}(\omega') - \theta_{n+1}(\omega) &> \Big(2\eta - \frac{5||v||_\infty + 2M}{\lambda}\Big) \Big(v(T^{n+1}\omega') - v(T^{n+1}\omega)\Big) \label{e.2.ii.a}\\
        &\geq \eta \Big(v(T^{n+1}\omega') - v(T^{n+1}\omega)\Big)\geq 0. \nonumber
    \end{align}
    \textit{Subcase b:} We claim that this subcase cannot arise by contradiction. Consider the function
    \begin{align*}
        f: (0,\pi) &\rightarrow \RR\\
        x &\mapsto \cot(x) + x
    \end{align*}
    which is a decreasing function. If $x<y$, then
    \[\cot(x) - \cot(y) > y - x.\]
    Consider the following computation, which uses the above line, inductive assumption, definition of $\lambda$, and assumptions on $v,g$.
    \begin{align*}
        0 &>\phi_{n+1}(\omega') - \phi_{n+1}(\omega)\\
        &\geq \lambda^2\Big(g(T^{n+1}\omega) \cot\theta_{n}(\omega) - g(T^{n+1}\omega') \cot\theta_{n}(\omega') \Big)\\
        &= \lambda^2 \Big(g(T^{n+1}\omega) (\cot\theta_{n}(\omega) - \cot\theta_{n}(\omega')) + \cot\theta_{n}(\omega') ( g(T^{n+1}\omega) -  g(T^{n+1}\omega') \Big)\\
        &\geq \lambda^2 \Big(g(T^{n+1}\omega) (\theta_{n}(\omega') - \theta_{n}(\omega)) - \cot\theta_{n}(\omega') ( g(T^{n+1}\omega') -  g(T^{n+1}\omega) \Big)\\
        &\geq \lambda^2 \Big(g(T^{n+1}\omega)\eta ( v(T^{n+1}\omega') - v(T^{n+1}\omega) )- \frac{1}{\lambda} ( g(T^{n+1}\omega') -  g(T^{n+1}\omega) \Big)\\
        &= \Bigg(g(T^{n+1}\omega)\lambda^2\eta\Bigg) \\
        &\hspace{8mm}\cdot\Bigg(\Big(v(T^{n+1}\omega') - v(T^{n+1}\omega)\Big) - \lambda\Big(t-v(T^{n+1}\omega')\Big)^2 - \Big(t-v(T^{n+1}\omega)\Big)^2\Bigg)\\
        &\geq \Bigg(\lambda(5||v||_\infty+2M)g(T^{n+1}\omega)\Bigg) \\
        &\hspace{8mm}\cdot \Bigg(\Big(v(T^{n+1}\omega') - v(T^{n+1}\omega)\Big)- \lambda \Big(v(T^{n+1}\omega') - v(T^{n+1}\omega)\Big)\Bigg) \\
        &\hspace{16mm}\cdot \Bigg(v(T^{n+1}\omega') + v(T^{n+1}\omega) - 2t\Bigg)\\
        &= \Bigg(v(T^{n+1}\omega') - v(T^{n+1}\omega)\Bigg)\\
        &\hspace{8mm}\cdot\Bigg(\lambda(5||v||_\infty+2M) g(T^{n+1}\omega) - \lambda (v(T^{n+1}\omega') + v(T^{n+1}\omega) - 2t)\Bigg)\\
        &\geq ( v(T^{n+1}\omega') - v(T^{n+1}\omega) ) \Big(\lambda(5||v||_\infty+2M) g(T^{n+1}\omega) - \lambda (4||v||_\infty+2M) \Big)\\
        &>0
    \end{align*}
    This is a contradiction to the assumption of $\phi_{n+1}(\omega) > \phi_{n+1}(\omega')$, so this case does not arise and we have completed Case $1$.
    
    For the second case, the analogous subcases are:
    \begin{enumerate}
        \item $g(T^{n+1}\omega')\cot\theta_{n}(\omega')\leq -\frac{1}{\lambda}$
        \item $0\geq g(T^{n+1}\omega')\cot\theta_{n}(\omega')> -\frac{1}{\lambda}$.
    \end{enumerate}
    Both follow a similar argument to the proofs provided in the respective subcases of the first case.
\end{proof}

\section{Subshifts of Finite Type}\hypertarget{SFT}{}

\subsection{Preliminaries}\hypertarget{prelim}{} Let $\A := \{0,1, \dots, \ell-1\}$ for some any positive integer $\ell\geq 2$. The full shift space is $\A^{\ZZ}$ and equipped with the left shift map $T$:
\[[T\omega]_n = \omega_{n+1} \text{ for all } n\in \ZZ.\]
We let $\A^\ZZ$ have a topology generated by the metric
\[d(\omega,\omega') = \begin{cases}
    e^{-N(\omega,\omega')} & \omega \neq \omega'\\
    0 & \omega = \omega'
\end{cases}\]
where $N: \A^\ZZ \times \A^\ZZ \rightarrow \NN$ defined as
\begin{equation}\label{e.N}
    N(\omega,\omega') = \max\{n\in \NN: \omega_j = \omega'_j \text{ for all } |j| < n\}.
\end{equation}

Cylinder sets are defined as
\[[n;i_0\dots i_{k-1}] := \{\omega: \omega_{n+l} = i_l \text{ for all } 0\leq l \leq k-1\}\]
where $n\in \ZZ$ and $i_0,\dots, i_k\in \A$. We say cylinder sets of the form above are of length $k$. If $\omega\in \A^\ZZ$ we may write $\underline{\omega}_k\in \A^k$ to represent the first $k$-coordinates, that is, $\underline{\omega}_n$ is equal to $\omega_0\dots \omega_{k-1}$. In summary, cylinder sets beginning at position $n$ with a length $k$ word can be written as $[n;\underline{i}_k]$.

Consider $\Omega \subset \A^\ZZ$ a subshift of finite type. That is, let $\F$ be a set of finitely many words of length $n$ that are forbidden, then $\Omega$ is the set of words that do not contain any forbidden words. Precisely,
\[\Omega:= \{\omega\in \A^\ZZ: \text{for all } j\in \ZZ, \omega_{j}\dots\omega_{j_n-1}\notin \F\}.\]
Every subshift defined in this way can also be encoded as a topological Markov chain. A topological Markov chain is defined as follows: consider a transition matrix $A$ of size $\ell\times \ell$ where entries $a_{i,j} \in \{0,1\}$ for $i,j\in \{1,\dots, \ell\}$. The associated topological Markov chain associated to $A$ is 
\[\Omega_A = \{\omega\in \A^\ZZ: a_{\omega_n,\omega_{n+1}} = 1 \text{ for all } n\in \ZZ\}.\]
It is apparent that $\Omega_A$ is a subshift of finite type with the family of banned words being $\F = \{ij\in \A^2: a_{i,j} = 0\}$, which consist only of words of length 2. Without loss of generality, we will always assume that $\F$, corresponding to $\Omega$, contains words of length 2.

Next, define 
\begin{align*}
    \Omega^+ &= \{(\omega_n)_{n \geq 0}: \omega\in \Omega\},\\
    \Omega^- &= \{(\omega_n)_{n \leq 0}: \omega\in \Omega\},
\end{align*}
which represent the set of positive and negative sequences and can be equipped with shift maps $T_\pm$. Denote cylinder sets in $\Omega^\pm$ as $[n;\underline{i}]^\pm$, respectively. Let $\pi^-$ and $\pi^+$ denote the canonical projections onto $\Omega^-$ and $\Omega^+$, respectively. Also, for each $j\in \A$ we define $\Omega_j = \Omega\cap [0;j]$ and $\Omega^\pm_j = \Omega^\pm \cap [0;j]^\pm$.

For each $\omega^-\in \Omega^-$ and $\omega^+\in \Omega^+$ we define
\begin{align*}
    W^{u}_{loc}(\omega^-) = \{\omega'\in \Omega: \omega'_j = \omega_j^-, \forall j\leq 0\},\\
    W^{s}_{loc}(\omega^+) = \{\omega'\in \Omega: \omega'_j = \omega_j^+, \forall j\geq 0\},
\end{align*}
which are called local unstable and stable sets respectively. Throughout, we will work over the unstable sets, but one can work over stable sets with an appropriate adjustment. We impose dictionary order on $W^u_{loc}(\omega^-)$, which means for $\omega,\omega'\in W^u_{loc}(\omega^-)$,
\[\omega < \omega' \iff \omega_{N(\omega, \omega')} < \omega'_{N(\omega,\omega')}.\]
Equality holds only when $\omega_j = \omega'_j$ for all $j\in \NN$. We collect some results that will be useful in the following section.

\begin{lemma}\label{l.resord}
    For every $n\in \NN$, then $T^n$ is monotone on cylinder sets $[0;\underline{i}]\subset W^u_{loc}(\omega^-)$ where $\underline{i}\in \A^n$ and $\omega^- \in \Omega^-$.
\end{lemma}
\begin{proof}
    Let $\omega\in  W^u_{loc}(\omega^-)$ and consider any any positive integer $n\geq 1$. Cylinders $[0;\underline{\omega}_n]\subset W^u_{loc}(\omega^-)$ are sets of the form
    \[[0;\underline{\omega}_n] = \{\omega'\in W^u_{loc}(\omega^-): \omega'_j = \omega_j \text{ for all } j\in \{0,\dots, n-1\}\}\]
    where $\omega_0 = \omega^-_0$. Let $\omega'\in [0;\underline{\omega}_n]$. Then note $N(\omega,\omega')\geq n$ and without loss of generality, assume that $\omega_{N(\omega,\omega')} < \omega'_{N(\omega,\omega')}$. The first element that $T^n\omega, T^n\omega'$ differ is precisely at the $N(\omega,\omega') - n$ position. It follows that $[T^n\omega]_{N(\omega,\omega') - n} < [T^n\omega']_{N(\omega,\omega') - n}$, since $[T^n\omega]_{N(\omega,\omega') - n} = \omega_{N(\omega,\omega')}$ and $[T^n\omega']_{N(\omega,\omega') - n} = \omega'_{N(\omega,\omega')}$. Thus, $T^n$ is monotone on cylinder sets $[0;\underline{i}]\subset W^u_{loc}(\omega^-)$ where $\underline{i}\in \A^n, i_0 = \omega^-_0,$ and $\omega^-\in \Omega$. 
\end{proof}

\begin{corollary}\label{c.compmonotone}
    If $v$ is monotone on $W^u_{loc}(\omega^-)$ and $n\in \NN$, then $v\circ T^n$ is monotone on cylinder sets $[0;\underline{i}]\subset W^u_{loc}(\omega^-)$ where $\underline{i}\in \A^n$ and $\omega^-\in \Omega^-$.
\end{corollary}

Next, we equip $\Omega$ with a measure $\mu$ satisfying the bounded distortion property. That is, there exists a $\gamma\geq 1$ such that for all cylinders $[n;j_0\dots j_k], [l; i_0\dots i_m] \subset \Omega$ where $l> n+k$ and $[n;j_0\dots j_k] \cap [l; i_0\dots i_m] \neq \varnothing$, we have
\[\gamma^{-1} \leq \frac{\mu([n;j_0\dots j_k]\cap [l; i_0\dots i_m])}{\mu([n;j_0\dots j_k])\cdot  \mu([l; i_0\dots i_m])} \leq \gamma.\]
Using the projections maps we can define measures $\mu^\pm = \pi^\pm_*\mu$ on $\Omega^\pm$, respectively. Additionally, for each $j\in \A$, we can define $\Omega_j = [0;j]$, $\mu_j = \mu|_{[0;j]},$ and $ \mu^\pm_j = \mu^\pm|_{[0;j]^\pm}$. We say the measure $\mu^+$ satisfies bounded distortion if there exists a $\gamma\geq 1$ such that for all $l,n,k\in \NN$ satisfying $l> n+k$ and $[n;j_0\dots j_k]^+ \cap [l; i_0\dots i_m]^+ \neq \varnothing$, we have
\[\gamma^{-1} \leq \frac{\mu^+([n;j_0\dots j_k]^+\cap [l; i_0\dots i_m]^+)}{\mu^+([n;j_0\dots j_k]^+)\cdot  \mu^+([l; i_0\dots i_m]^+)} \leq \gamma.\]
One can define bounded distortion for $\mu^-$ in a synonymous manner. 

To conclude this section, we collect some properties that arise from a measure $\mu$ with bounded distortion property. The first of which implies that we may always assume, without loss of generality, that $\mu^+$ has bounded distortion property.

\begin{lemma}\label{l.mupmbdp}
    A measure $\mu$ has bounded distortion if and only if $\mu^+$ or $\mu^-$ has bounded distortion.
\end{lemma}
\begin{proof}
    For the forward direction, it suffices to show that $\mu^+$ has bounded distortion property. Let $l,n,k\in \NN$ satisfy $l> n+k$ and $[n;j_0\dots j_k]^+ \cap [l; i_0\dots i_m]^+ \neq \varnothing$. Notice
    \begin{align*}
        \mu^+([n;j_0\dots j_k]^+) &= \mu((\pi^+)^{-1}[n;j_0\dots j_k]^+)\\
        &= \mu\Bigg(\Big\{\omega\in \Omega: \omega_{n+i} = j_i \text{ for all } i\in \{0,\dots, k\}\Big\} \Bigg)\\
        &= \mu([n;j_0\dots j_k])
    \end{align*}
    and by the same argument we find:
    \begin{align*}
        \mu^+([l; i_0\dots i_m]^+) &= \mu([l; i_0\dots i_m]),\\
        \mu^+([n;j_0\dots j_k]^+\cap [l; i_0\dots i_m]^+) &= \mu([n;j_0\dots j_k]\cap [l; i_0\dots i_m]).
    \end{align*}
    Putting everything together,
    \[\frac{\mu^+([n;j_0\dots j_k]^+\cap [l; i_0\dots i_m]^+)}{\mu^+([n;j_0\dots j_k]^+)\mu^+([l; i_0\dots i_m]^+)} = \frac{\mu([n;j_0\dots j_k]\cap [l; i_0\dots i_m])}{\mu([n;j_0\dots j_k])\mu([l; i_0\dots i_m])}\]
    and by bounded distortion property of $\mu$ we conclude the following:
    \[\gamma^{-1}\leq \frac{\mu^+([n;j_0\dots j_k]^+\cap [l; i_0\dots i_m]^+)}{\mu^+([n;j_0\dots j_k]^+)\mu^+([l; i_0\dots i_m]^+)}\leq \gamma.\]

    For the reverse direction, let $\mu^+$ have bounded distortion property. Let 
    \[[n;j_0\dots j_k], [l; i_0\dots i_m] \subset \Omega\] where $l> n+k$ and $[n;j_0\dots j_k] \cap [l; i_0\dots i_m] \neq \varnothing$. By $T$-invariance of $\mu$ we can assume that $n\in \NN$. By the argument above but in reverse, we can conclude $\mu$ has bounded distortion.
\end{proof}

Another outcome of a measure with bounded distortion property is that all such measures have local product structure. That is, there exists a measurable function $\psi: \Omega \rightarrow \RR_+$ such that $\psi\in L^1(\Omega_j, \mu_j^-\times\mu^+_j)$ and
\[d\mu_j = \psi d(\mu_j^-\times\mu^+_j)\]
for each $j\in \A$. Moreover, the measurable function $\psi$ is positive and bounded $\mu$ almost everywhere.

\begin{lemma}\label{l.BDLP_LRNbound}
    If $\mu$ has bounded distortion property, then $\mu^+_j$ is equivalent to $\mu^-_j\times\mu^+_j$ for each $j\in \A$. Consequently, $\mu$ has local product structure and there exists a $\kappa\geq 1$ such that
    \[\kappa^{-1}\leq \psi \leq \kappa\]
    $\mu$ almost everywhere.
\end{lemma}
\begin{proof}
    Consider any cylinder $[-n; j_{-n}\dots j_{-1}j_0j_1\dots j_m]\subset \Omega$ where $n,m\in \NN$. In order, we will use the definition of product measure, definition of $\mu_j$, subadditivity, and bounded distortion.
    \begin{align}
        \mu_{j_0}^-\times\mu_{j_0}^+&([-n; j_{-n}\dots j_{-1}j_0j_1\dots j_m]) \label{e.test1}\\
        &= \mu_{j_0}^-([-n; j_{-n}\dots j_{-1}j_0]^-) \mu_{j_0}^+([0; j_{0}\dots j_{m-1}j_m]^+) \nonumber\\
        &= \mu_{j_0}([-n; j_{-n}\dots j_{-1}j_0]) \mu_{j_0}([0; j_{0}\dots j_{m-1}j_m])\nonumber\\
        &\leq \mu_{j_0}([-n; j_{-n}\dots j_{-1}]) \mu_{j_0}([0; j_{0}\dots j_{m-1}j_m])\nonumber\\
        &\leq \gamma \mu_{j_0}([-n; j_{-n}\dots j_{-1}j_{0}\dots j_{m-1}j_m])\nonumber
    \end{align}
    The approximation holds for Borel sets, as every borel set can be approximated with cylinder sets. Let $E\subset [-n; j_{-n}\dots j_{-1}j_{0}\dots j_{m-1}j_m]$ such that
    \[\mu_{j_0}([-n; j_{-n}\dots j_{-1}j_{0}\dots j_{m-1}j_m])\leq \mu_{j_0}(E) + \epsilon.\]
    Combined with the previous argument, we find
    \begin{align}
        \mu_{j_0}^-\times\mu_{j_0}^+(E) &\leq \mu_{j_0}^-\times\mu_{j_0}^+([-n; j_{-n}\dots j_{-1}j_{0}\dots j_{m-1}j_m]) \nonumber\\
        &\leq \gamma\mu_{j_0}([-n; j_{-n}\dots j_{-1}j_{0}\dots j_{m-1}j_m]) \nonumber\\
        &\leq \gamma\mu_{j_0}(E) + C\epsilon. \nonumber
    \end{align}
    Letting $\epsilon \rightarrow 0$, we conclude that $\mu_{j_0}^-\times\mu_{j_0}^+ << \mu_{j_0}$. 

    To show $\mu_{j_0} << \mu_{j_0}^-\times\mu_{j_0}^+$, consider
    \begin{align}
        \mu_{j_0}([-n;& j_{-n}\dots j_{-1}j_0j_1\dots j_m]) \label{e.test2}\\
        &\leq \gamma \mu_{j_0}^-([-n; j_{-n}\dots j_{-1}j_0]^-)\mu^+_{j_0}([1;j_1\dots j_m]^+) \nonumber\\
        &\leq \frac{\gamma^2}{\mu^+_{j_0}([0;j_0]^+)}\mu_{j_0}^-([-n; j_{-n}\dots j_{-1}j_0]^-)\mu^+_{j_0}([0;j_0j_1\dots j_m]^+) \nonumber\\
        &\leq \kappa \mu_{j_0}^-\times\mu_{j_0}^+([-n; j_{-n}\dots j_{-1}j_0j_1\dots j_m])\nonumber
    \end{align}
    where $\kappa = \max\{\frac{\gamma^2}{\mu^+_{j_0}([0;j])}: j\in \A\}$. Thus, the measures $\mu_{j_0}$ are equivalent $\mu_{j_0}^-\times\mu_{j_0}^+$.
    
    The equivalence holds for all $j_0\in \A$, so let $\psi_{j_0}: \Omega_{j_0} \rightarrow \RR$ be the Lebesgue-Radon-Nikodym derivative when $\mu_{j_0}^-\times\mu_{j_0}^+ << \mu_{j_0}$, which is in $L^1(\Omega_{j_0})$. Define $\psi = \sum_{j\in \A} \psi_j$, which is $L^1(\Omega_{j})$ for each $j\in \A$. Moreover, 
    \[d\mu_j = \psi d(\mu^-_j\times \mu^+_j),\]
    which holds for all $j\in \A$.
    
    The equivalence of measures implies that the Lebesgue-Radon-Nikodym derivative for $\mu_{j_0} << \mu_{j_0}^-\times\mu_{j_0}^+$ must be $\frac{1}{\psi_{j_0}}$. Hence, $\psi > 0$ for $\mu$ almost everywhere. By (\ref{e.test1}) and (\ref{e.test2})
    \begin{equation}\label{e.measequivbound}
        \kappa^{-1}\mu_{j_0}^-\times\mu_{j_0}^+(E)\leq \mu_{j_0}(E)\leq \kappa\mu_{j_0}^-\times\mu_{j_0}^+(E),
    \end{equation}
    which implies
    \[\kappa^{-1}\leq \psi \leq \kappa\]
    $\mu$ almost everywhere. Indeed, if $\psi > \kappa$ for a set $E\subset \Omega_j$ of positive measure for some $j\in \A$, then by local product structure,
    \[\mu_j(E) = \int_E \psi d(\mu^-_j\times \mu^+_j) > \kappa (\mu^-_j\times \mu^+_j)(E),\]
    which is a contradiction to (\ref{e.measequivbound}). The lower bound follows an analogous argument.
\end{proof}

\subsection{Half-Line}

For this section, we work in the half-line setting in which we slightly abuse notation and write $\Omega^+$ as the half-line, $T_+$ as the half-line left shift map, and $d$ as a metric on $\Omega^+$. One can similarly define order on $\Omega^+$ and have the same results as Lemma \ref{l.resord} and Corollary \ref{c.compmonotone}. Recall that we assume that $v$ satisfies $0\leq v(\omega)\leq 1$ for all $\omega\in \Omega^+$ and $\I = [-1,2]$. Also, from Section \hyperlink{secprelim}{2}, it suffices to show positivity on the cocycle admitted from polar decomposition to prove Theorem \ref{t.halfPLE}. Thus, we show the following theorem.

\begin{theorem}\label{t.halfPLEpd}
    For the same assumptions as in Theorem \ref{t.halfPLE}, there exists a $C_1 = C_1(v)>0$ such that for any $\lambda> 0$, we have
    \[\Tilde{L}(t;\lambda) > \log\lambda - C_1 \text{ for all } t\in \I.\]
\end{theorem}

We show that $T_+$ satisfies the hyperbolic injective property. Indeed, for every $n\in \NN$ we can define partitions 
\[\P_n := \{[0;\underline{\omega}]^+: \underline{\omega}\in \A^n\}.\]
It follows that for each $n\in \NN$, then
\[\Omega^+ = \bigsqcup_{S\in \P_n} S.\]
Moreover, for any $\omega\in \A^{n+1}$, one has $[0;\underline{\omega}_{n+1}]^+ \subset [0;\underline{\omega}_n]^+$, so $\P_{n+1}$ is a refinement of $\P_n$. Let $S\in \P_n$. If $\omega,\omega'\in S$ such that $T^n_+(\omega) = T^n_+(\omega')$, then $\omega_j = \omega'_j$ for all $j\in \NN$, so $\omega = \omega'$. Hence, $T^n_+|_S$ injects into $\Omega^+$ for all $S\in \P_n$. Applying Lemma \ref{l.key} on the compact set $\I$ and using the anti $\alpha$-H\"older condition implies following corollary.

\begin{corollary}\label{c.thetarelatecylinder}
    Consider a potential $v$ that is anti $\alpha$-H\"older, continuous, and monotone with respect to dictionary order all on $\Omega^+$, and satisfies $0\leq v(\omega)\leq 1$ for all $\omega\in \Omega$. For any $[0:\underline{i}_n]^+\in \P_n$ and $\omega,\omega'\in [0;\underline{i}_n]^+$, then one has the inequality
    \begin{equation}
        |\theta_n(\omega', t) - \theta_n(\omega, t)|\geq \frac{H}{20}d(T^n_+\omega,T^n_+\omega')^\alpha,
    \end{equation}
    for any $n\in \NN$ and $t\in \I$.
\end{corollary}
 
Next, we turn our attention to step (iii). First, we introduce the necessary notation. For every $\underline{i}\in \A^n$, define $m^{\underline{i}} \in [0;\underline{i}]^+$ for each $\underline{i}\in \A^n$ as:
\[m_j^{\underline{i}} = \begin{cases}
    i_j & j = \{0,\dots, n-1\}\\
    0 & \text{ otherwise }
\end{cases}\]
Similarly, define $M^{\underline{i}} \in [0;\underline{i}]$ as:
\[M_j^{\underline{i}} = \begin{cases}
    i_j & j = \{0,\dots, n-1\}\\
    \ell-1 & \text{ otherwise }
\end{cases}\]
Observe from our ordering that $m^{\underline{i}}$ and $M^{\underline{i}}$ represent the minimum and maximum element of the cylinder $[0;\underline{i}]^+$, respectively. Define $\C_{-1, \cdot} = \varnothing$ and for every $n\in \NN$ and $\underline{i}\in \A^n$ we define $\C_{n,\underline{i}}$ as follows: the collection of $\beta\in [0;\underline{i}]^+$ such that there exists a $k_\beta\in \NN$ so that 
\[|\theta_n(\beta) - \frac{\pi}{2} - k_\beta\pi| = \inf_{\omega\in \Omega^+}|\theta_n(\omega) - \frac{\pi}{2} - k_\beta\pi|\] 
where 
\[\theta_n(m^{\underline{i}})\leq \frac{\pi}{2} + k\pi \leq \theta_n(M^{\underline{i}_n})\}.\]
The existence of $\beta$ follows by compactness of $[0;\underline{i}]^+$. The set $\C_{n,\underline{i}}$ contains all elements in step $n$ inside the cylinder $[0;\underline{i}]^+$such that the image under $\theta_n$ is in the set $\frac{\pi}{2}+\pi\ZZ$. In particular, these are points where cancellation can occur, as this is when the rotation matrix rotates by angle close to or equal to $\frac{\pi}{2}$. Let
\begin{equation}\label{e.definecn}
    \C_n = \bigsqcup_{\underline{i}\in \A^n} \C_{n,\underline{i}}.
\end{equation}

\begin{lemma}\label{l.upperboundofbadpoints}
    For each $n\in \NN$,
    \[Card(\C_{n,\underline{i}}) \leq Card(\C_{n - 1, \underline{i}}) + 1\]
    for any $\underline{i}\in \A^n$. 
\end{lemma}
\begin{proof}
    Let $n = 0$ and notice that $|\theta_0(\Omega^+)| < \pi$, so by monotonicity of $\theta_0$ on $\Omega^+$ there is at most one $\frac{\pi}{2}$ point contained in the image. Thus, it follows that
    \[Card(\C_0) \leq 1\]
    which completes the base case.
    
    For induction, notice 
    \begin{equation}\label{e.relatephitheta}
        \{\omega\in \Omega^+: \phi_{n+1}(\omega)\in \frac{\pi}{2}+\pi\ZZ\} = \{\omega\in \Omega^+: \theta_n(\omega)\in \frac{\pi}{2} + \pi\ZZ\}.
    \end{equation}
    We show that for any $\underline{i}\in \A^n$ and $a\in \A$,
    \[Card(\C_{n+1, \underline{i}a}) \leq Card(\C_{n,\underline{i}a}) + 1.\]
    Indeed, $T^{n+1}_+: [0;\underline{i}a]^+ \rightarrow \Omega^+$ is an isomorphism, so
    \begin{align*}
        |\theta_{n+1}([0;\underline{i}a]^+)| &\leq \phi_{n+1}([0;\underline{i}a]^+) + \theta(T^{n+1}_+[0;\underline{i}a])\\
        &\leq \phi_{n+1}([0;\underline{i}a]^+) + |\theta(\Omega)|\\
        &< \phi_{n+1}([0;\underline{i}a]^+) + \pi.
    \end{align*}
    By monotonicity of $\theta_{n+1}|_{[0;\underline{i}a]^+}$ and (\ref{e.relatephitheta}), 
    \[Card(\C_{n+1, \underline{i}a}) \leq Card(\C_{n,\underline{i}a}) + 1,\]
    completing the proof.
\end{proof}

Moving to step (iv), we relate $\delta$-balls around $\theta_n(\beta)$ where $\beta\in \C_n$, to a cylinder set of the form $[0;\underline{\beta}_{n+N}]$ for some fixed $N\in \NN$. To begin, we extract a maximal ratio $p$ of cylinder sets differing by one letter to attain a decay rate of cylinder sets.
\begin{lemma}\label{l.boundratio}
    There exists a positive constant $B > 0$ such that
    \[B^{-1} \leq \min_{a,b\in \A}\Bigg\{\frac{\mu^+([0;\underline{i}a]^+)}{\mu^+([0;\underline{i}b]^+)}\Bigg\} \leq \max_{a,b\in \A}\Bigg\{\frac{\mu^+([0;\underline{i}a]^+)}{\mu^+([0;\underline{i}b]^+)}\Bigg\} \leq B\]
    for any $n\in \NN$ and $\underline{i}\in \A^n$.
\end{lemma}
\begin{proof}
    Pick any $B_0> 1$ such that
    \[B_0^{-1} \leq \min_{a,b\in \A}\Bigg\{\frac{\mu^+([0;a]^+)}{\mu^+([0;b]^+)}\Bigg\} \leq \max_{a,b\in \A}\Bigg\{\frac{\mu^+([0;a]^+)}{\mu^+([0;b]^+)}\Bigg\}\leq B_0\]
    which exists since there are a finite number of ratios considered.
    
    Consider $n\in \NN$ and any $\underline{i}\in \A^n$. By bounded distortion property and $T$-invariance,
    \begin{equation}\label{e.bdpmaxrat}
        \gamma^{-1}\leq \frac{\mu^+([0;\underline{i}k]^+)}{\mu^+([0;\underline{i}]^+)\mu^+([0;k]^+)} \leq \gamma
    \end{equation}
    for any $k\in \A$. Applying (\ref{e.bdpmaxrat}) on any $a,b\in \A$, we find the following upper bound.
    \begin{align*}
        \frac{\mu^+([0;\underline{i}a]^+)}{\mu^+([0;\underline{i}b]^+)} &\leq \frac{\gamma \mu^+([0;\underline{i}]^+)\mu^+([0;a]^+)}{\gamma^{-1}\mu^+([0;\underline{i}]^+)\mu^+([0;b]^+)}\\
        &= \gamma^2 \frac{\mu^+([0;a]^+)}{\mu^+([0;b]^+)}\\
        &\leq \gamma^2 B_0
    \end{align*}
    Similarly, we find the following lower bound.
    \begin{align*}
        \frac{\mu^+([0;\underline{i}a]^+)}{\mu^+([0;\underline{i}b]^+)} &\geq \frac{\gamma^{-1} \mu^+([0;\underline{i}]^+)\mu^+([0;a])}{\gamma\mu^+([0;\underline{i}]^+)\mu^+([0;b]^+)}\\
        &=\gamma^{-2} \frac{\mu^+([0;a]^+)}{\mu^+([0;b]^+)}\\
        &\geq \gamma^{-2} B_0^{-1}
    \end{align*}
    Letting $B = \max\{\gamma^2 B_0, |\A|\} \geq 2$ completes the proof.
\end{proof}

Let $p = \frac{B}{B+1}$, which will be fixed for the remainder of this section. It is evident that for any $n\in \NN$ and $\underline{i}\in \A^n$
\[\mu^+([0;\underline{i}]) = \sum_{a\in \A} \mu^+([0;\underline{i}a]).\]
Consider any $b\in \A$ and by Lemma \ref{l.boundratio},
\[\frac{\mu^+([0;\underline{i}])}{\mu^+([0;\underline{i}b])} = 1 + \sum_{a\in \A\setminus \{b\}} \frac{\mu^+([0;\underline{i}a])}{\mu^+([0;\underline{i}b])} \geq 1 + \frac{1}{B}.\]
Rearranging, we may conclude for any $\underline{i}\in \A^n$
\begin{equation}\label{e.decayrate}
    \mu^+([0;\underline{i}b]^+) \leq p\mu^+([0;\underline{i}]^+)
\end{equation}
for all $b\in \A$. Combining (\ref{e.decayrate}) and Corollary \ref{c.thetarelatecylinder} we are able to relate the $\delta$-balls to cylinder sets through the following lemma.

\begin{lemma}\label{l.subsets}
    For any $\delta\in \Big(0, \min\{\frac{\pi}{8}, \frac{H}{5^\alpha \cdot 40}\}\Big)$, there exists a positive integer $N = N(\delta)\geq 1$ such that for all $n\in \NN, \underline{i}\in \A^n$, and  $\beta\in \C_{n,\underline{i}}$, \begin{align}
        \{\omega&\in [0;\underline{i}]^+: |\theta_n(\omega) - \frac{\pi}{2} -k_{\beta}\pi|<\delta \}\label{e.relatesubsets}\\
        &\subset \{\omega\in [0;\underline{\beta}_n]^+: |\theta_n(\omega) - \theta_n(\beta)| <2\delta\} \nonumber\\
        &\subset [0;\underline{\beta}_{n + N}]^+. \nonumber
    \end{align}
    Moreover, there is some fixed constant $C_2 = C_2(v) > 0$ such that 
    \[p^{N} < C_2\delta^{-\ln p}.\]
\end{lemma}
\begin{proof}
    Let $\delta\in \Big(0, \min\{\frac{\pi}{8}, \frac{H}{5^\alpha \cdot 40}\}\Big)$, $n\in \NN, \underline{i}\in \A^n$, and $\beta\in \C_{n,\underline{i}}$. Notice if
    \[\inf_{\omega\in [0;\underline{i}]}|\theta_n(\omega) - \frac{\pi}{2} - k_\beta\pi| > \delta,\]
    then the first set in (\ref{e.relatesubsets}) is empty. In this case, the statement holds trivially for any $N\geq 1$. It remains to consider when
    \[\inf_{\omega\in [0;\underline{i}]}|\theta_n(\omega) - \frac{\pi}{2} - k_\beta\pi| < \delta.\]
    
    Consider any $\omega\in \{\omega\in [0;\underline{i}]^+: |\theta_n(\omega) - \frac{\pi}{2} -k_{\beta}\pi|<\delta \}$. Then
    \[|\theta_n(\omega) - \theta_n(\beta)| \leq |\theta_n(\omega) - \frac{\pi}{2} -k_\beta \pi| + |\frac{\pi}{2} + k_\beta \pi - \theta_n(\beta)| < 2\delta.\]
    This implies that the first containment of (\ref{e.relatesubsets}) holds. Next, we show that $\omega$ and $\beta$ lie in the same $\pi\ZZ$ interval. Indulging in two straightforward computations, one attains:
    \[\theta_n(\beta) + 2\delta < \frac{\pi}{2} + k_\beta\pi + 3\delta < \frac{7\pi}{8} + k_\beta \pi < (k_\beta+1)\pi\]
    and
    \[\theta_n(\beta) - 2\delta > \frac{\pi}{2} + k_\beta\pi - 3\delta > \frac{\pi}{8} + k_\beta \pi > k_\beta\pi.\]
    Applying Corollary \ref{c.thetarelatecylinder},
    \begin{align*}
        \Big\{\omega\in [0;\underline{\beta}_n]^+: |\theta_n(\omega) - \theta_n(\beta)| < 2\delta\Big\} 
        &\subset \Big\{\omega\in [0;\underline{\beta}_n]^+: d(T^n_+\omega,T^n_+\beta) < \Big(\frac{40\delta}{H}\Big)^{\frac{1}{\alpha}}\Big\}.
    \end{align*}
    Recall the function $N(\cdot, \cdot)$ from (\ref{e.N}). If \[\omega \in \Big\{\omega\in [0;\underline{\beta}_n]^+: d(T^n_+\omega,T^n_+\beta) < \Big(\frac{40\delta}{H}\Big)^{\frac{1}{\alpha}}\Big\},\]
    then it follows that
    \[N(T^n_+\omega, T^n_+\beta) > -\ln\Bigg(\Big(\frac{40}{H}\Big)^{\frac{1}{\alpha}} \delta \Bigg).\]
    Let $N = \Bigg\lfloor -\ln\Bigg(\Big(\frac{40}{H}\Big)^{\frac{1}{\alpha}} \delta \Bigg) \Bigg\rfloor$ and note
    \[N = \Bigg\lfloor-\ln\Bigg(\Big(\frac{40}{H}\Big)^{\frac{1}{\alpha}} \delta  \Bigg)\Bigg\rfloor \geq \Bigg\lfloor-\ln\Bigg(\frac{1}{5} \Bigg)\Bigg\rfloor \geq 1.\]
    With this choice of $N$ we obtain
    \[\Big\{\omega\in [0;\underline{\beta}_n]^+: d(T^n_+\omega,T^n_+\beta) < \Big(\frac{40\delta}{H}\Big)^{\frac{1}{\alpha}}\Big\} \subset [0;\underline{\beta}_{n+N}]^+\]
    and $p^{N} < C_2\delta^{-\ln p}$ where $C_2 = \Big(\frac{40}{H}\Big)^{\frac{1}{\alpha}}$. 
\end{proof}

Define $||\cdot||_{\RR\PP^1}$ to output the distance to the nearest $\pi \ZZ$. Consider any $\delta\in\Big(0, \min\{\frac{\pi}{8}, \frac{H}{5^\alpha \cdot 40}\}\Big)$ and define
\[D_n(\delta) = \{\omega: ||\theta_n(\omega) - \frac{\pi}{2}||_{\RR\PP^1} < \delta\}\]
and a subset of $D_n(\delta)$ as
\[B_n(\delta) = \bigsqcup_{\beta \in \C_n} \{\omega\in [0;\underline{\beta}_n]^+: |\theta_n(\omega) -\theta_n(\beta)| < \delta\}\]
for all $n\in \NN$. To complete step (iv), one must control the measure of $D_n(\delta)$, which is partially controlled by the $n$-step geometric series of $p$. 

\begin{lemma}\label{l.boundsum}
    For any $n\in \NN$,
    \[\sum_{\beta\in \C_n} \mu^+([0;\underline{\beta}_n]^+) \leq \frac{1-p^{n+1}}{1-p}.\]
    Consequently 
    \[\sum_{\beta\in \C_n} \mu^+([0;\underline{\beta}_n]^+) \leq \frac{1}{1-p}\]
    for all $n\in \NN$.
\end{lemma}
\begin{proof}
    The base case of $n=0$ is trivial, since $Card(\C_0)\leq 1$ and $\mu(\Omega)\leq 1$. Assume for induction that there exists some $n\in \NN$, such that
    \[\sum_{\beta\in \C_n} \mu^+([0;\underline{\beta}_n]^+) \leq \frac{1-p^{n+1}}{1-p}.\]
    Applying Lemma \ref{l.upperboundofbadpoints}, 
    \begin{align*}
        \sum_{\beta\in \C_{n+1}} \mu^+([0;\underline{\beta}_{n+1}]^+) &= \sum_{\underline{i}\in \A^{n+1}} Card(\C_{n+1,\underline{i}}) \mu^+([0;\underline{i}]^+)\\
        &\leq \sum_{\underline{i}\in \A^{n+1}} \Big(Card(\C_{n,\underline{i}}) + 1\Big) \mu^+([0;\underline{i}]^+)\\
        &= \sum_{\underline{i}\in \A^{n+1}} Card(\C_{n,\underline{i}})\mu^+([0;\underline{i}]^+) +  \sum_{\underline{i}\in \A^{n+1}} \mu^+([0;\underline{i}]^+)\\
        &\leq \sum_{\underline{i}\in \A^{n+1}} Card(\C_{n,\underline{i}})\mu^+([0;\underline{i}_n]^+)p + 1\\
        &= p \sum_{\beta\in \C_n} \mu^+([0;\underline{\beta}_n]^+) +1\\
        &\leq \frac{1-p^{n+2}}{1-p}.
    \end{align*}
    This completes the inductive step. The latter part is immediate, as
    \[\frac{1-p^{n+1}}{1-p}\leq \frac{1}{1-p}\]
    for all $n\in \NN$.
\end{proof}

\begin{lemma}\label{l.thetacylinder}
    For any $\delta\in \Big(0, \min\{\frac{\pi}{8}, \frac{H}{5^\alpha \cdot 40}\}\Big)$,   
    \[\mu^+(B_n(\delta)) < C\delta^{-\ln p}\]
    for all $n\in \NN$.
\end{lemma}
\begin{proof}
    Let $\delta\in \Big(0, \min\{\frac{\pi}{8}, \frac{H}{5^\alpha \cdot 40}\}\Big)$. By Corollary \ref{l.subsets}, there exists positive integer $N\geq 1$ such that 
    \[p^{N} < C_2\delta^{-\ln p}\]
    and
    \[\{\omega\in [0;\underline{\beta}_n]^+: |\theta_n(\omega) -\theta_n(\beta)| < \delta\} \subset [0;\underline{\beta}_{n+N}]^+\]
    for all $n\in \NN$ and $\beta\in \C_n$. As a consequence, for any $n\in \NN$ we have the containment
    \[B_n(\delta) \subset \bigsqcup_{\beta\in \C_n}[0;\underline{\beta}_{n+N}]^+.\]
    Combining Lemma \ref{l.subsets}, Lemma \ref{l.boundsum}, (\ref{e.decayrate}), bounded distortion, and $T_+$-invariance leads to the following:
    \begin{align*}
        \mu^+(B_n(\delta)) &= \mu^+(\bigsqcup_{\beta\in \C_n}[0;\underline{\beta}_{n+N}]^+)\\
        &= \sum_{\beta \in \C_n}\mu^+([0;\underline{\beta}_{n+N}]^+)\\
        &\leq C\sum_{\beta \in \C_n}\mu^+([0;\underline{\beta}_{n}]^+)\mu^+([0;\beta_{n}\dots\beta_{n+N-1}]^+)\\
        &\leq Cp^{N} \sum_{\beta \in \C_n}\mu^+([0;\underline{\beta}_{n}]^+)\\
        &< C\delta^{-\ln p} \sum_{\beta \in \C_n}\mu^+([0;\underline{\beta}_{n}]^+)\\
        &< C\delta^{-\ln p}.
    \end{align*}
\end{proof}

To measure $D_n(\delta)\setminus B_n(\delta)$, we consider any $\underline{i}\in \A^n$ and $\Big(D_n(\delta)\setminus B_n(\delta)\Big)\cap [0;\underline{i}]^+$. It can be observed that the elements in this set arise when $\theta_n(m^{\underline{i}_n})$ or $ \theta_n(M^{\underline{i}_n})$ are $\delta$-close to a $\frac{\pi}{2} + \pi\ZZ$. Thus, using similar arguments in measuring $B_n(\delta)$, we find that $\Big(D_n(\delta)\setminus B_n(\delta)\Big)\cap [0;\underline{i}]^+$ is contained in the union of cylinders around $m^{\underline{i}_n}$ and $M^{\underline{i}_n}$. 

\begin{lemma}\label{l.Dntocylinder}
    For $\delta\in \Big(0, \min\{\frac{\pi}{8}, \frac{H}{5^\alpha \cdot 40}\}\Big)$, there exists positive integer $N = N(\delta) \geq 1$ such that
    \[\Bigg(D_n(\delta) \setminus B_n(\delta)\Bigg) \cap [0;\underline{i}]^+\subset [0;\underline{m^{\underline{i}}}_{n+N}]^+ \sqcup [0;\underline{M^{\underline{i}}}_{n+N}]^+\]
    for every $n\in \NN$ and $\underline{i}\in \A^n$
\end{lemma}
\begin{proof}
    Let $\delta\in \Big(0, \min\{\frac{\pi}{8}, \frac{H}{5^\alpha \cdot 40}\}\Big)$ and pick $N$ to be the same as in Lemma \ref{l.subsets}. Consider any $n\in \NN$ and any $\underline{i}\in \A^n$. If $Card(\C_n\cap [0;\underline{i}]^+) = k$ for some $k\in \NN$, then by monotonicity there exists some $\alpha\in \NN$ such that
    \begin{equation}\label{e.mMbounds}
        \frac{\pi}{2} + \alpha\pi < \theta_n(m^{\underline{i}}) \leq \theta_n(\omega) \leq \theta_n(M^{\underline{i}}) < \frac{\pi}{2} + (\alpha + k+1)\pi
    \end{equation}
    for all $\omega\in [0;\underline{i}]^+$. Let $\eta_1 = \delta + (\frac{\pi}{2} + \alpha\pi - \theta_n(m^{\underline{i}}))$, $\eta_2 = \delta + (\theta_n(M^{\underline{i}})- \frac{\pi}{2} - (\alpha+ k+ 1) \pi)$.  Assume that $\eta_1$ and $\eta_2$ are positive. By (\ref{e.mMbounds}) and our choice of $\eta_1$ and $\eta_2$, the following holds:
    \begin{align*}
        \Bigg(D_n(\delta) \setminus& B_n(\delta)\Bigg) \cap [0;\underline{i}]^+ \\
        &= \{\omega: |\theta_n(\omega) - \frac{\pi}{2} - \alpha\pi|<\delta\} \sqcup \{\omega: |\theta_n(\omega) - \frac{\pi}{2} - (\alpha+ k+ 1)\pi|<\delta\}\\
        &\subset  \{\omega: |\theta_n(\omega) - \theta_n(m^{\underline{i}})|<\eta_1\} \sqcup \{\omega: |\theta_n(\omega) -\theta_n(M^{\underline{i}})|<\eta_2\}
    \end{align*}
    By the same argument in Lemma \ref{l.subsets},
    \begin{align*}
         \{\omega: |\theta_n(\omega) - \theta_n(m^{\underline{i}})|<\eta_1\} &\subset [0;\underline{m^{\underline{i}}}_{n+N}]^+,\\
         \{\omega: |\theta_n(\omega) -\theta_n(M^{\underline{i}})|<\eta_2\} &\subset [0;\underline{M^{\underline{i}}}_{n+N}]^+.
    \end{align*}
    Putting everything together, we have
    \[\Bigg(D_n(\delta) \setminus B_n(\delta)\Bigg) \cap [0;\underline{i}]^+ \subset [0;\underline{m^{\underline{i}}}_{n+N}]^+ \sqcup [0;\underline{M^{\underline{i}}}_{n+N}]^+.\]
    If $\eta_1$ and $\eta_2$ were negative, then 
    \[\Bigg(D_n(\delta) \setminus B_n(\delta)\Bigg) \cap [0;\underline{i}]^+ = \varnothing.\] 
    If only one $\eta_j$ was positive, then $\Big(D_n(\delta)\setminus B_n(\delta)\Big)\cap [0;\underline{i}]^+$ is a subset of either $[0;\underline{m^{\underline{i}}}_{n+N}]^+$ or $[0;\underline{m^{\underline{i}}}_{n+N}]^+$, depending on which $\eta_j$ was positive.
\end{proof}

\begin{corollary}\label{c.measureDN}
    For any $\delta\in \Big(0, \min\{\frac{\pi}{8}, \frac{H}{5^\alpha \cdot 40}\}\Big)$,
    \[\mu^+(D_n(\delta)) < C\delta^{-\ln p}\]
    for all $n\in \NN$.
\end{corollary}
\begin{proof}
    By Lemma \ref{l.thetacylinder} and Lemma \ref{l.Dntocylinder} we find:
    \begin{align*}
        \mu^+&(D_n(\delta)) \\
        &< C\delta^{-\ln p} + \sum_{\underline{i}\in \A^n} \Bigg(\mu^+([0;\underline{m^{\underline{i}}}_{n+N}]^+) + \mu^+([0;\underline{M}^{\underline{i}}_{n+N}]^+)\Bigg)\\
        &\leq C\delta^{-\ln p} \\
        &\hspace{4mm}+ \gamma\sum_{\underline{i}\in \A^n} \mu^+([0;\underline{i}]^+)\Bigg(\mu^+([0;m^{\underline{i}}_n\dots m^{\underline{i}}_{n+N}]^+) + \mu^+([0;M^{\underline{i}}_n\dots M^{\underline{i}}_{n+N}]^+)\Bigg)\\
        &\leq C\delta^{-\ln p} + 2C\delta^{-\ln p}\\
        &< C\delta^{-\ln p}.
    \end{align*}
\end{proof}

By a similar argument in \cite{young}, we conclude positive Lyapunov exponents.

\begin{proof}[Proof of Theorem \ref{t.halfPLEpd}]
    Let $\vec{e}_1, \vec{e}_2$ represent the standard basis of $\RR^2$ and 
    \[\Tilde{\Lambda}(T_+\omega^+) = \begin{pmatrix}
        \lambda \sqrt{g(T_+\omega^+, t)} & 0 \\
        0 & \frac{1}{\lambda} \sqrt{\frac{1}{g(T_+\omega^+, t)}}
    \end{pmatrix}.\] 
    Observe 
    \[\lim_{n\rightarrow \infty} \frac{1}{n}\int_{\Omega^+} \log||A_n(\omega^+)|| d\mu^+(\omega^+) \geq \lim_{n\rightarrow \infty} \frac{1}{n}\int_{\Omega^+} \log||A_n(\omega^+)\vec{e}_1|| d\mu^+(\omega^+)\]
    and existence follows from the Osceledets Multiplicative Ergodic Theorem. Let $\omega^+\in \Omega^+$. Then define
    \[\vec{w}_n(\omega^+) = A_n(\omega^+)\vec{e}_1 \text{ and } \vec{v}_n(\omega^+) = R_{\theta(T^n\omega^+)}A_n(\omega^+)\vec{e}_1.\]
    Rewriting, we find
    \[\vec{w}_{n+1}(\omega^+) = \Tilde{\Lambda}(T^{n+1}_+\omega^+) \vec{v}_{n}(\omega^+) \text{ and } \vec{v}_n(\omega^+) = R_{\theta(T^n_+\omega^+)}\vec{w}_n(\omega^+).\]
    The previous line implies
    \[||\vec{v}_n(\omega^+)|| = ||\vec{w}_n(\omega^+)||.\]
    By definition of $\theta_n(\omega^+)$, we find
    \[\vec{v}_n(\omega^+) = \cos\theta_n(\omega^+)||\vec{v}_n(\omega^+)||\vec{e}_1 + \sin\theta_n(\omega^+)||\vec{v}_n(\omega^+)||\vec{e}_2.\]
    Inductively, 
    \begin{align*}
        ||\vec{w}_n(\omega^+)|| &= ||\Tilde{\Lambda}(T^{n}_+\omega^+) \vec{v}_{n-1}||\\
        &\geq ||\Tilde{\Lambda}(T^{n}_+\omega^+)\cos\theta_{n-1}(\omega^+)||\vec{v}_{n-1}(\omega^+)||\vec{e}_1||\\
        &\geq \lambda|\cos\theta_{n-1}(\omega^+)| \cdot ||\vec{v}_{n-1}(\omega^+)||\\
        &= \lambda|\cos\theta_{n-1}(\omega^+)|\cdot ||\vec{w}_{n-1}(\omega^+)||\\
        &\geq \lambda^n \prod_{j=0}^{n-1}|\cos\theta_j(\omega^+)|.
    \end{align*}
    Applying $\log$ on both sides,
    \begin{align*}
        \log||A_n(\omega^+)\vec{e}_1|| &= \log||\vec{w}_n(\omega^+)||\\
        &\geq \log(\lambda^n \prod_{j=0}^{n-1}|\cos\theta_j(\omega^+)|) \\
        &= n\log \lambda + \sum_{j=0}^{n-1} \log|\cos\theta_j(\omega^+)| 
    \end{align*}
    and it follows that
    \begin{align}
        \lim_{n\rightarrow \infty} \int_{\Omega^+}& \frac{1}{n}\log ||A_n(\omega^+)\vec{e}_1|| \mu^+(\omega^+)\nonumber\\
        &\geq \log \lambda + \limsup_{n\rightarrow \infty}\frac{1}{n}\sum_{j=0}^{n-1} \int_{\Omega^+} \log|\cos\theta_j(\omega^+)| d\mu^+(\omega^+).\label{boundthis}
    \end{align}
    From (\ref{boundthis}), it suffices to find a lower bound of $\int_{\Omega^+} \log|\cos\theta_j(\omega^+)| dm(\omega^+)$.
    Fix any $k \in \{0,\dots, n-1\}$. The following estimates will be independent of $k$. Define for all $i\in \NN$,
    \[J_i := \{\omega^+\in \Omega^+: ||\theta_i(\omega^+) - \frac{\pi}{2}||_{\RR\PP^1}< \frac{\delta}{2^i}\}.\]
    By Corollary \ref{c.measureDN},
    \[\mu^+(J_i) < C\Big(\frac{\delta}{2^i}\Big)^{- \ln p}.\]
    A direct computation gives
    \begin{align*}
        \int_{J_0} \log|\cos\theta_k(\omega^+)|d\mu^+(\omega^+) &= \int_{J_0} \log|\sin(\theta_k(\omega^+) - \frac{\pi}{2})|d\mu^+(\omega^+) \\
        &\geq \int_{J_0} \log\frac{2}{\pi}|\theta_k(\omega^+) - \frac{\pi}{2}|d\mu^+(\omega^+) \\
        &= \sum_{i\in \NN} \int_{J_i\setminus J_{i+1}} \log\frac{2}{\pi}|\theta_k(\omega^+) - \frac{\pi}{2}|d\mu^+(\omega^+)\\
        &\geq \sum_{i\in \NN} \int_{J_i\setminus J_{i+1}} \log\Bigg(\frac{2}{\pi}\Big(\frac{\delta}{2^{i+1}}\Big)^{- \ln p}\Bigg)d\mu^+(\omega^+)\\
        &= \sum_{i\in \NN} \mu^+(J_i\setminus J_{i+1}) \log\Bigg(\frac{2}{\pi}\Big(\frac{\delta}{2^{i+1}}\Big)^{- \ln p}\Bigg)\\ 
        &\geq \sum_{i\in \NN}  C\Big(\frac{\delta}{2^{i+1}}\Big)^{- \ln p}\log\Bigg(\frac{2}{\pi}\Big(\frac{\delta}{2^{i+1}}\Big)^{- \ln p}\Bigg)\\
        &= -C\delta^{-\ln p} \sum_{i\in \NN} \frac{\ln p}{2^{-(i+1)\ln p}} \log\Bigg(\Big(\frac{2}{\pi}\Big)^{-\frac{1}{\ln(p)}}\Big(\frac{\delta}{2^{i+1}}\Big)\Bigg) \\
        &= -C\delta^{-\ln p} \sum_{i\in \NN} \frac{i+1}{2^{-(i+1)\ln p}} \log\Bigg(\Big(\frac{2}{\pi}\Big)^{-\frac{i+1}{\ln(p)}}\Big(\frac{\delta^{\frac{1}{i+1}}}{2}\Big)\Bigg) \\
        &\geq -C\delta^{-\ln p}\sum_{i\in \NN} \frac{i}{(2^{-\ln p})^i}\\
        &\geq -C\delta^{-\ln p}.
    \end{align*}
    Consider
    \[J_0^c = \{\omega^+\in \Omega^+: ||\theta_k(\omega^+)-\frac{\pi}{2}||\geq  \delta\}.\]
    It follows that
    \begin{align*}
        \int_{J_0^c} \log|\cos\theta_k(\omega^+)|dm(\omega^+) &= \int_{J_0^c} \log|\sin(\theta_k(\omega^+) - \frac{\pi}{2})|d\mu^+(\omega^+)\\
        &\geq \int_{J_0} \log\frac{2}{\pi}|\theta_k(\omega^+) - \frac{\pi}{2}|d\mu^+(\omega^+)\\
        &\geq \int_{J_0} \log(\frac{2}{\pi}\Big(\frac{\delta}{2^i}\Big)^{- \ln p})d\mu^+(\omega^+)\\
        &> C\log\delta^{-\ln p}.
    \end{align*}
    Pick $m\in \NN$ such that $\frac{1}{m} < \delta^{-\ln p} < \Big(\min\{\frac{\pi}{8}, \frac{H}{5^\alpha\cdot 40}\}\Big)^{-\ln p}$, which provides
    \begin{align*}
        \int_{\Omega^+} \log|\cos\theta_k(\omega^+)|d\mu^+(\omega^+) &= \int_{J_0^c} \log|\cos\theta_k(\omega^+)|d\mu^+(\omega^+)\\
        &\hspace{8mm}+ \int_{J_0} \log|\cos\theta_k(\omega^+)|d\mu^+(\omega^+)\\
        &> -\frac{1}{m}C - C\log m\\
        &> -C.
    \end{align*}
    Therefore, 
    \[\frac{1}{n}\sum_{j=0}^{n-1} \int_{\Omega^+} \log|\cos\theta_j(\omega^+)| d\mu^+(\omega^+) > -C\]
    for all $n\in \NN$. The preceding estimate implies
    \[\lim_{n\rightarrow \infty} \frac{1}{n}\int_{\Omega^+} \log ||A_n(\omega^+)\vec{e}_1|| d\mu^+(\omega^+) > \log \lambda - C_0 = \log\lambda - C_0,\]
    which completes the proof of Theorem \ref{t.halfPLEpd}.
\end{proof}

\subsection{Full-Line}\hypertarget{fullline}{}

To attain a full-line result, we must first reduce the full-line to the half-line. We do so by thinking of the local unstable sets as a map on $\Omega^-$. That is,
\[W^u_{loc}(\omega^-) = \{\omega\in \Omega: \pi^-(\omega) = \omega^-\},\]
which provides the following partition
\[\Omega = \bigsqcup_{\omega^- \in\Omega^-} W^u_{loc}(\omega^-).\]
The partition leads to a disintegration of $\mu$ along the local unstable sets where existence follows from Rohklin's disintegration theorem. A disintegration along local unstable sets is a family of probability measures $\{\mu_{\omega^-}^u\}_{\omega^-\in\Omega^-}$ where each is supported on $W^u_{loc}(\omega^-)$ and satisfies
\[\mu(D) = \int_{\Omega^-}\mu_{\omega^-}^u(D) d\mu^-(\omega^-)\]
for each measurable $D\subset \Omega$. The disintegration implies that given $f\in L^1(\mu)$ one obtains the following equality:
\[\int_\Omega f(\omega)d\mu(\omega) = \int_{\Omega^-} \int_{W^u_{loc}(\omega^-)} f(\omega)d\mu^u_{\omega^-}(\omega) d\mu^-(\omega^-).\]
In particular, if $f(\omega) = \frac{1}{n}\log||A_n(\omega)||$:
\begin{align}
    \int_\Omega &\frac{1}{n}\log||A_n(\omega)||d\mu(\omega) \label{reducefull}\\
    &= \int_{\Omega^-} \Bigg(\int_{W^u_{loc}(\omega^-)}\frac{1}{n}\log||A_n(\omega)||d\mu^u_{\omega^-}(\omega) \Bigg) d\mu^-(\omega^-).\nonumber
\end{align}

For each $\omega^-\in \Omega^-$, define
\begin{align*}
    I_{\omega^-}: \Omega^+_{\omega^-_0} &\rightarrow W^u_{loc}(\omega^-)\\
    \omega^+ &\mapsto(\omega^-, \omega^+)
\end{align*}
which is an isomorphism. From the isomorphism it is apparent that the inner integral from (\ref{reducefull}) is essentially the half-line. At this point, one would like to directly apply the half-line setting, but it is not necessarily true that the disintegrated measures have $T_+$-invariance. Thus, the half-line theorem cannot be applied. To resolve the predicament, we show that the disintegrated measures $\mu^u_{\omega^-}$ are equivalent to a $T_+$-invariant measure, in particular $(I_{\omega^-})_*\mu^+_j$. Moreover, the Lebesgue-Radon-Nikodym derivative between the two is precisely $\psi$ from local product structure of $\mu$. 

\begin{lemma}\label{l.IpushBD}
    For each $j\in \A$, if $\mu^+_j$ has bounded distortion property, then $(I_{\omega^-})_*\mu^+_j$ has bounded distortion property for all $\omega^- \in \Omega^-_j$.
\end{lemma}
\begin{proof}
    Use the same argument as Lemma \ref{l.mupmbdp} and replace $\pi^+$, $\mu^+$, $\mu$ with $I_{\omega^-}$, $(I_{\omega^-})_*\mu^+_j$, $\mu^+_j$.
\end{proof}

\begin{lemma}\label{l.ujbdp}
    For each $j\in \A$, 
    \[d\mu^u_{\omega^-} = \psi d((I_{\omega^-})_*\mu^+_j)\]
    holds for $\mu^-_j$ almost every $\omega^- \in \Omega^-_j$ and $\psi$ is the same function from local product structure of $\mu$.
\end{lemma}
\begin{proof}
    Fix any $j\in \A$. Local product structure of $\mu$ implies
    \[\mu_j(E) = \int_{\Omega^-_j}\int_{\Omega^+_j} \psi(\omega^-, \omega^+) \chi_E d\mu^+_j(\omega^+) d\mu^-_j(\omega^-),\]
    for any measurable set $E\subset \Omega_j$. On the other hand, by Rohklin's Disintegration Theorem
    \[\mu_j(E) = \int_{\Omega^-_j} \mu^u_{\omega^-,j}(E) d\mu^-_j = \int_{\Omega^-_j} \int_{W^u_{loc}(\omega^-)} \chi_E d\mu^u_{\omega^-} d\mu^-_j.\]
    Consider the measure $(I_{\omega^-})_*\mu^+_j$, as in Lemma \ref{l.IpushBD}. It follows that
    \begin{align*}
        \int_{\Omega^-_j}\int_{\Omega^+_j}& \psi(\omega^-, \omega^+) \chi_E d\mu^+_j(\omega^+) d\mu^-_j(\omega^-) \\
        &= \int_{\Omega^-_j}\int_{W^u_{loc}(\omega^-)} \psi(\omega^-, \omega^+) \chi_E d((I_{\omega^-})_*\mu^+_j)(\omega^+) d\mu^-_j(\omega^-).
    \end{align*}
    Notice 
    \[\nu_{\omega^-}: E\mapsto \int_{W^u_{loc}(\omega^-)} \psi(\omega^-,\cdot) \chi_E(\omega^-,\cdot) d((I_{\omega^-})_*\mu^+_j)(\omega^+)\]
    is a measure. The two preceding equations leads to
    \[\mu_j(E) = \int_{\Omega^-_j} \nu_{\omega^-}(E) d\mu^-.\]
    Thus, $\nu_{\omega^-} = \mu^u_{\omega^-}$ for $\mu^-_j$ almost every $\omega^- \in \Omega^-_j$, by uniqueness of Rohklin's Disintegration Theorem. Hence, the following equality holds
    \[d\mu^u_{\omega^-} = \psi d((I_{\omega^-})_*\mu^+_j).\]
\end{proof}

\begin{corollary}\label{c.plestable}
    Let $v$ be uniformly anti $\alpha$-H\"older, continuous, bounded, and monotone with respect to dictionary order all on $W^u_{loc}(\omega^-)$ for all $\omega^-\in \Omega^-$. Then there exists a $C_0 > 0$ and $\kappa = \kappa(\mu)\geq 1$ such that for any $\lambda>0$ and any $\omega^-\in \Omega^-$, we have
    \[\liminf_{n\rightarrow \infty} \frac{1}{n} \int_{W^u_{loc}(\omega^-)} \log||A^E_n(\omega)|| d\mu^u_{\omega^-}(\omega) > \kappa^{-1}\log \lambda - C_0\]
    for all $E\in \RR$.
\end{corollary}
\begin{proof}
    By the isomorphism $I_{\omega^-}$, Lemma \ref{l.BDLP_LRNbound}, and Lemma \ref{l.ujbdp}, one has
    \begin{align*}
        \liminf_{n\rightarrow \infty}& \frac{1}{n} \int_{W^u_{loc}(\omega^-)} \log||A^E_n(\omega)|| d\mu^u_{\omega^-}(\omega)\\
        &=  \liminf_{n\rightarrow \infty} \frac{1}{n} \int_{W^u_{loc}(\omega^-)} \psi(\omega)\log||A^E_n(\omega)|| d(I_{\omega^-})_*\mu^+_j(\omega) \\
        &\geq  \kappa^{-1} \liminf_{n\rightarrow \infty} \frac{1}{n} \int_{\Omega^+_{\omega^-_0}} \log||A^E_n(I_{\omega^-}(\omega^+))|| d\mu^+_j(\omega^+) \\
        &= \kappa^{-1} \liminf_{n\rightarrow \infty} \frac{1}{n} \int_{\Omega^+_{\omega^-_0}} \log||A^E_n(\omega^-,\omega^+)|| d\mu^+_j(\omega^+) \\
        &> \kappa^{-1}\log\lambda - \kappa^{-1} C_1\\
        &= \kappa^{-1}\log\lambda - C_0
    \end{align*}
    where $C_0 = \kappa^{-1}C_1$.
\end{proof}

In Corollary \ref{c.plestable}, we consider the $\liminf$ because $\mu^u_{\omega^-}$ does not have $T$-invariance, so the limit may not exist. With Corollary \ref{c.plestable}, we can conclude Theorem \ref{t.fullPLE}.

\begin{proof}[Proof of Theorem \ref{t.fullPLE}]
    Combining (\ref{reducefull}), Fatou Lemma, and Corollary \ref{c.plestable}, one can conclude positivity as follows.
    \begin{align*}
        \lim_{n\rightarrow \infty}& \frac{1}{n}\int_\Omega \log||A_n(\omega)||d\mu(\omega) \\
        &\geq \liminf_{n\rightarrow \infty} \int_{\Omega^-} \frac{1}{n}\Bigg(\int_{W^u_{loc}(\omega^-)}\log||A_n(\omega)||d\mu^u_{\omega^-}(\omega) \Bigg) d\mu^-(\omega^-)\\
        &\geq \int_{\Omega^-} \liminf_{n\rightarrow \infty} \frac{1}{n}\Bigg(\int_{W^u_{loc}(\omega^-)}\log||A_n(\omega)||d\mu^u_{\omega^-}(\omega) \Bigg) d\mu^-(\omega^-)\\
        &> \int_{\Omega^-} \Big(\kappa^{-1}\log\lambda - C_0\Big) d\mu^-(\omega^-)\\
        &= \kappa^{-1}\log\lambda - C_0
    \end{align*}
\end{proof}

\section{One Expanding Direction on Tori}\hypertarget{OEDT}{}

In this section, we prove Theorem \ref{t.PLEhd}. In particular, we show that tori of arbitrary integer dimension $d\geq 1$ with a map with at least one expanding direction leads to the Schr\"odinger cocycle admitting positive Lyapunov exponents. We begin with the positive integers $d\geq 2$, as $d=1$ will become a consequence of the argument. The first step is to reduce the dimension of the problem to dimension 1. To do so, we construct a set $Z$ of dimension $d-1$ such that every $z\in Z$ admits a unique line segment $W_z$. The union of all such line segments is equal to the fundamental domain $F$ of $\TT^d$. We provide a geometric picture for $d = 3$.

\begin{figure}[H]
    \centering
    \begin{tikzpicture}[3d view={120}{20}, line join=round, line cap=round, declare function={a=3;}]
    
        \draw[fill=blue!30, opacity=0.6] 
                (3/2,a,0) -- (3/2,0,0) -- (3/2,0,a) -- (3/2,a,a) -- cycle;
        
        \draw[red, thick, ->] (-0.5, 1/2, 1/2) -- (a, 1/2, 1/2) ;
        \draw[red, thick, ->] (-0.5, 1/2, 1) -- (a, 1/2, 1) ;
        \draw[red, thick, ->] (-0.5, 1/2, 3/4) -- (a, 1/2, 3/4) ;
        \draw[red, thick, ->] (-0.5, 1/2, 5/4) -- (a, 1/2, 5/4) ;
        \draw[red, thick, ->] (-0.5, 1/2, 6/4) -- (a, 1/2, 6/4) ;
        \draw[red, thick, ->] (-0.5, 1/2, 7/4) -- (a, 1/2, 7/4) ;
        \draw[dashed, gray] (0,0,0) -- (a,0,0);
        \draw[dashed, gray] (0,0,0) -- (0,a,0);
        \draw[dashed, gray] (0,0,0) -- (0,0,a);
            
        \draw[thick] (a,0,0) -- (a,a,0) -- (0,a,0);
        \draw[thick] (a,0,0) -- (a,0,a) -- (a,a,a) -- (0,a,a) -- (0,0,a);
        \draw[thick] (a,a,0) -- (a,a,a);
        \draw[thick] (0,a,0) -- (0,a,a);
        \draw[thick] (a,0,a) -- (0,0,a);
            
    \end{tikzpicture}
        
    \caption{Geometric picture representing the local unstable leaves (red, $W_z$) and the plane parameterizing the leaves (blue, $Z$) for the matrix $\begin{pmatrix}
        2 & 0 & 0\\
        0 & 1 & 0\\
        0 & 0 & 1
    \end{pmatrix}$ acting on $\TT^3$. Note that the choice of $Z$ is not unique.}
\end{figure}
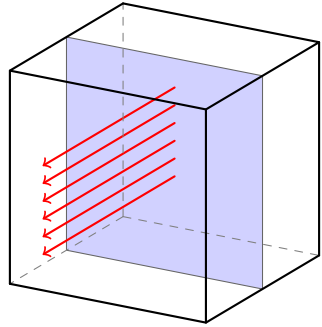

Begin by considering any positive integer $d\geq 1$ and fix any $M\in G(d,\ZZ)$. Let $u$ denote an associated unit eigenvector to $\beta$. Define the map 
\begin{align*}
    T_M: \RR^d &\rightarrow \RR^d\\
    \omega &\mapsto M\omega
\end{align*}
which, slightly abusing notation, induces a map $T_M: \TT^d\rightarrow \TT^d$. Equip $\TT^d$ with $d$-dimensional Lebesgue measure $\mu$. To construct $Z$ we introduce the following notation
\[\sgn(x) = \begin{cases}
    1 & x>0\\
    0 & x=0\\
    -1 & x<0
\end{cases}\]
and 
\[I_{\sgn(x)} = \begin{cases}
    [0,1) & x = 0, 1\\
    (-1,0] & x = -1
\end{cases}.\]
Let the fundamental domain of $\TT^d$ be
\[F:= \prod_{i=1}^d I_{\sgn(u_i)},\]
which is chosen for convenience in the following lemma. 

\begin{lemma}\label{l.hyperplane}
    Consider a positive integer $d\geq 2$ and assume $u\in\RR^d$ is a unit vector. Then there exists an affine hyperplane $P\subset \RR^d$ such that when we define
    \[Z = P \cap F,\]
    the following statement holds: For every $\omega\in F$, there exists a unique $z\in Z$ such that $\omega\in $Image$(f_z)\cap F$ where
    \[f_z(t) = tu + z.\]
\end{lemma}
\begin{proof}
    Without loss of generality, we may assume
    \[u_i = \begin{cases}
        > 0 & i = 1,\dots, k\\
        < 0 & i = k+1, \dots, j\\
        = 0 & i = j+1, \dots, d
    \end{cases}\]
    for $0\leq k,j\leq d$ where $0\leq k + j \leq d$. Note one can permute the $1,\dots, d$ subscripts to match the original vector. 

    Define an affine hyperplane 
    \[P:= \{x\in \RR^d: \sum_{i=1}^k x_i - \sum_{i=k+1}^{k+j} x_i = 1\}.\]
    Let $Z = P\cap F$. We claim this choice of $Z$ suffices. Let $\omega\in F$ and define
    \[g(\omega)= \sum_{i=1}^k \omega_i - \sum_{i=k+1}^{k+j} \omega_i.\]
    Notice that $g(u) \neq 0$, since if $g(u) = 0$, then 
    \[\sum_{i=1}^k \omega_i = \sum_{i=k+1}^{k+j} \omega_i.\]
    The left hand side is strictly positive and the right is strictly negative, which is a contradiction. 

    Define $t_0$ as
    \[t_0 = \frac{g(\omega) - 1}{g(u)}\]
    and $z \in \RR^d$ with components defined as
    \[z_i = \begin{cases}
        \omega_i -t_0u_i & i = 1,\dots, k+j\\
        0 & i = k+j + 1, \dots, d
    \end{cases}.\]
    Observe
    \begin{align*}
        g(z) &= \sum_{i=1}^k z_i - \sum_{i=k+1}^{k+j} z_i\\
        &= \sum_{i=1}^k (\omega_i -tu_i) - \sum_{i=k+1}^{k+j} (\omega_i -tu_i)\\
        &= g(\omega) - tg(u)\\
        &= g(\omega) - (g(\omega)-1)\\
        &= 1,
    \end{align*}
    which implies that $z\in P$. 
    
    It remains to show that $z\in F$, so consider three cases:
    \begin{enumerate}
        \item $\omega\in P$,
        \item $\omega$ satisfies $g(\omega)>1$,
        \item $\omega$ satisfies $g(\omega)<1$.
    \end{enumerate}
    Case 1 is trivial as $z = \omega$. Let us consider case $2$. Assume for a contradiction that there exists a $j = 1,\dots, k$ such that $z_j < 0$. Define $\Tilde{t} = \frac{\omega_j}{u_j}$. Since $z_j < 0$,
    \[z_j = \omega_j - t_0 u_j < 0 \implies \frac{\omega_j}{u_j} < t_0 \implies \Tilde{t} < t_0.\] 
    But $\omega_j - tu_j$ is a decreasing function and $w_j -\Tilde{t}u_j = 0$, so $\Tilde{t}> t_0$, which is a contradiction.
    
    The argument for $i = k+1, \dots, k+j$ is the same as above and there is nothing to show when $i = k+j+1, \dots, d$. Also, notice if $z\in Z$, then $f_z(t_0) = \omega$ and hence $\omega\in$ Image$(f_z)\cap F$. The proof for case 3 has an analogous argument to case 2.
\end{proof}

From this point on we fix $Z$ to represent the affine hyperplane, as in Lemma \ref{l.hyperplane}. Define for each $z\in Z$ a local unstable leaf
\[W_z = \text{Image}(f_z) \cap F.\] 
It follows that
\[F = \bigsqcup_{z\in Z} W_z\]
and after taking the quotient
\[\TT^d = \bigsqcup_{z\in Z} W_z.\]
By Rohklin's disintegration theorem, one attains the equality (\ref{reducefull}), where the conditional measures are normalized Lebesgue measure and denoted as $m_z$ for each $z\in Z$. On this note, it remains to show positive Lyapunov exponents along $W_z$ with the measure $m_z$.

To find a suitable partition, recall that $W_z$ is isometric to a line segment in the direction of $u$ in $\RR^d$, which induces order on $W_z$. First, define a set that contains the endpoints of $W_z$, $\D_{0,z} = \{w_0, w_1\}$. For each $n\in \NN$ we observe that $T^n_MW_z$ is of finite length, since each iterate of $T_M$ increases the length of the line segment by a factor of $\beta$. Consequently, there are a finite number of times the line crosses the fundamental domain after each iteration of $T_M$. Let $\D_{n,z}$ be the collection of all such points in $W_z$ at step $n$, then define the set
\[\D_{n,z}' = \bigcup_{j=0}^n \D_{n,z}.\]
By the argument in \cite[Lemma 2]{chiem} one finds that $Card(\D_{n,z}')<\sum_{j=0}^{n+1} \beta^{k+j+1}$ for some fixed $k\in \NN$ which is chosen to satisfy $2\sqrt{d}(d+2)<\beta^k$. With the induced order, we can label and order all the points in $\D_{n,z}'$ as
\[w_{n,0} < \dots < w_{n,Card(\D_{n,z}')-2}.\]
Define $I_{n,z,j} = [w_{n,j}, w_{n,j+1})$ for $0\leq j\leq Card(\D_{n,z}')-2$. Remark that for every $n\in \NN$, the points $\omega_{n,0}, \omega_{n,Card(\D_{n,z}')-2}$ are the end points of $W_z$. For each $n\in \NN$, define the partition
\[\P_{n,z} = \{I_{n,z,j}\}_j.\]
By construction, $\P_{n+1,z}$ is a refinement of $\P_{n,z}$, since $\D_{n,z}' \subset \D_{n+1,z}'$. The hyperbolic injective property is satisfied, as $T^n_MI_{n,z,j} \subset W_z$ for some $z\in Z$ by construction of $I_{n,z,j}$. By applying Lemma \ref{l.key} on $\I$ provides the following corollary.

\begin{corollary}\label{c.keycat}
    Consider a potential $v$ that is uniformly anti $\alpha$-H\"older, continuous, and monotone all along $W_z$ for all $z\in Z$, and satisfies $0\leq v(\omega)\leq 1$ for all $\omega \in \TT^d$. Then there exists a $\lambda_0 = \lambda_0(v)>0$ such that $\theta_n$ is monotone on every $I_{n,z,j}\in \P_{n,z}$ for any $z\in Z$. Moreover, for any $z\in Z$, $I_{n,z,j}\in \P_{n,z}$, and $\omega,\omega'\in I_{n,z,j}$, then one has the inequality
    \[|\theta_n(\omega', t)-\theta_n(\omega, t)|\geq \frac{H}{20}\beta^{n\alpha}|\omega-\omega'|^\alpha,\]
    for any $n\in \NN$ and $t\in \I$.
\end{corollary}

With this estimate at hand, we can run steps (iii) - (v). Defining $\C_n$ as in (\ref{e.definecn}) with cylinder sets replaced with $I_{n,z,j}$ and using the estimate of $Card(B_n)$, one attains the estimate
\begin{equation}\label{e.CARD2x}
    Card(\C_n) < \sum_{j=0}^{n+2}\beta^{k+j+1}.
\end{equation}

On the other hand, (iv) follows similarly to the subshift of finite type setting. Define for any $\delta>0$,
\begin{align*}
    D_{n,z}(\delta) &= \{\omega\in W_z: ||\theta_n(\omega) - \frac{\pi}{2}||_{\RR\PP^1} < \delta\},\\
    B_{n,z}(\delta)&= \bigsqcup_{\omega\in \C_n} \{\Tilde{\omega}\in W_z: |\theta_n(\omega) - \theta_n(\Tilde{\omega})| < \delta\}.
\end{align*}
By a similar argument in Lemma \ref{l.thetacylinder}, we obtain an identical estimate on $D_{n,z}(\delta)$.

\begin{corollary}\label{c.measureDNdoubling}
    For any $\delta \in(0,\frac{\pi}{2})$,  
    \[m_z(D_{n,z}(\delta))<C\delta^{\frac{1}{\alpha}}\]
    for all $n\in \NN$ and $z\in Z$.
\end{corollary}
\begin{proof}
    Consider any $z\in Z$. By Corollary \ref{c.keycat}, for each $\omega\in \C_n$
    \[\{\Tilde{\omega}: |\theta_n(\omega) - \theta_n(\Tilde{\omega})|<\delta\} \subset \{\Tilde{\omega}: |\omega- \Tilde{\omega}| < C\frac{\delta^{\frac{1}{\alpha}}}{\beta^n}\},\]
    which implies
    \begin{equation}\label{e.omegacn}
        m_z(\{\Tilde{\omega}: |\theta_n(\omega) - \theta_n(\Tilde{\omega})|<\delta\}) < C\frac{\delta^{\frac{1}{\alpha}}}{\beta^n}.
    \end{equation}
    Note that (\ref{e.omegacn}) holds for every $\omega\in \C_n$ and combined with (\ref{e.CARD2x}),
    \begin{align}
        m_z(B_{n,z}(\delta)) &\leq \sum_{\omega\in \C_n} m_z(\{\Tilde{\omega}: |\theta_n(\omega) - \theta_n(\Tilde{\omega})|<\delta\}) \nonumber\\
        &< \sum_{\omega\in \C_n} \frac{C\delta^{\frac{1}{\alpha}}}{\beta^n} \nonumber\\
        &= \frac{C\delta^{\frac{1}{\alpha}}}{\beta^n} Card(\C_n)\nonumber\\
        &\leq \frac{C\delta^{\frac{1}{\alpha}}}{\beta^n}\sum_{j=0}^{n+2}\beta^{k+j+1} \nonumber\\
        &\leq C\delta^{\frac{1}{\alpha}} (\beta + \beta^2 + \beta^3 + \sum_{j=0}^{n-1} \beta^{-j}) \nonumber\\
        &\leq C\delta^{\frac{1}{\alpha}} \label{e.cmeasuredn.1}.
    \end{align}
    Similar to Lemma \ref{l.Dntocylinder}, for any $n\in \NN$ and $0\leq j \leq Card(\D_{n,z}')-2$ the elements in $\Big(D_{n,z}(\delta) \setminus B_{n,z}(\delta)\Big)\cap I_{n,z,j}$ arise from the endpoints of $W_z$. It follows that
    \begin{equation*}\label{e.cmeasuredn.2}
        m_z\Bigg(\Big(D_n(\delta)\setminus B_n(\delta)\Big) \cap I_{n,j}\Bigg) < \frac{C\delta^{\frac{1}{\alpha}}}{\beta^n}.
    \end{equation*}
    Putting the estimates (\ref{e.cmeasuredn.1}) and (\ref{e.cmeasuredn.2}) together, we obtain
    \begin{align*}
        m_z(D_{n,z}(\delta)) &\leq m_z(B_{n,z}(\delta)) + \sum_{j=0}^{Card(\D_{n,z}') -2}m_z\Bigg(\Big(D_n(\delta)\setminus B_n(\delta)\Big) \cap I_{n,j}\Bigg) \\
        &< C\delta^{\frac{1}{\alpha}} + \beta^{n+k}\frac{C\delta^{\frac{1}{\alpha}}}{\beta^n} \\
        &<  C\delta^{\frac{1}{\alpha}}.
    \end{align*}
\end{proof}

To conclude positivity along $W_z$, one runs an identical argument provided in the half-line or see \cite{young}. To conclude positivity on $\TT^d$ one uses Rohklin's disintegration, Lebesgue's Dominated convergence theorem, and the lower estimate of the Lyapunov exponent along $W_z$. See \cite[Proposition 1]{chiem} for details.

For $d=1$ one can define a sequence of partitions in a similar manner to $W_z$, then apply Lemma \ref{l.key}. Afterwards, one runs the same argument provided after Corollary \ref{c.keycat} to conclude positivity, which concludes Theorem \ref{t.PLEhd}.

\end{document}